\documentclass[letterpaper, 12 pt]{article}
\usepackage{amsthm, amsmath, amssymb, amsfonts}
\usepackage{xy}\xyoption{all}

\newcommand{\Flag}{\mathcal{F}\ell}
\newcommand{\Elts}{\mathbb{E}} 
\newcommand{\ZZ}{\mathbb{Z}}
\newcommand{\QQ}{\mathbb{Q}}
\newcommand{\RR}{\mathbb{R}}
\newcommand{\CC}{\mathbb{C}}
\newcommand{\PP}{\mathbb{P}}
\renewcommand{\AA}{\mathbb{A}}
\renewcommand{\P}{\PP}
\newcommand{\Pow}{\mathcal{P}}

\newcommand{\Poly}{\mathtt{Poly}}
\newcommand{\Cone}{\mathtt{Cone}}
\DeclareMathOperator{\hilb}{\mathtt{hilb}}
\newcommand{\Sym}{\mathrm{Sym}}

\newcommand{\into}{\hookrightarrow}

\newcommand{\isom}{\cong}
\newcommand{\bw}{{\textstyle\bigwedge}} 

\DeclareMathOperator{\Span}{Span}

\DeclareMathOperator{\conv}{conv}

\newcommand{\Char}{\mathrm{Char}}
\DeclareMathOperator{\Frac}{\mathrm{Frac}}

\newcommand{\pt}{\mathtt{pt}}
\newcommand{\KTP}{K^0_T}
\newcommand{\eqv}[1]{#1^T} 

\newcommand{\Mat}{\mathrm{Mat}}

\newcommand{\GL}{\mathrm{GL}}

\DeclareMathOperator{\Spec}{Spec}

\newcommand{\Hom}{\mathrm{Hom}}

\newcommand{\Fl}{\Flag}

\newcommand{\OO}{\mathcal{O}}
\renewcommand{\O}{\OO}
\newcommand{\One}{\boldsymbol{1}}

\newcommand{\zzeta}{\boldsymbol{\zeta}}

\newcommand{\SM}{\mathrm{SM}}
\newcommand{\IndMat}{\mathcal I}

\newtheorem{conj}{Conjecture}[section]
\newtheorem{theorem}[conj]{Theorem}
\newtheorem{prop}[conj]{Proposition}
\newtheorem{lemma}[conj]{Lemma}
\newtheorem{cor}[conj]{Corollary}
\newtheorem{question}[conj]{Question}

\newtheorem*{theoremtutte}{Theorem \ref{TutteTheoremProved}}
\newtheorem*{theoremh}{Theorems \ref{DiagonalPolynomial}, \ref{OldNew}}

\theoremstyle{definition}

\newtheorem{remark}[conj]{Remark}
\newtheorem{example}[conj]{Example}
\newtheorem{Condition}[conj]{Condition}

\newcommand{\newword}[1]{\textbf{\emph{#1}}}

\begin{document}

\title{$K$-classes for matroids and equivariant localization}

\author{Alex Fink \\ David E Speyer}

\date{}

\maketitle

\begin{abstract}
To every matroid, we associate a class in the $K$-theory of the Grassmannian.
We study this class using the method of equivariant localization.
In particular, we provide a geometric interpretation of the Tutte polynomial.
We also extend results of the second author concerning the behavior of such classes under direct sum, series and parallel connection and two-sum;
these results were previously only established for realizable matroids, and their earlier proofs were more difficult.
\end{abstract}

\section{Introduction}\label{sec:intro}

Let $H_1$, $H_2$, \dots, $H_n$ be a collection of hyperplanes through the origin in $\CC^d$. 
The study of such hyperplane arrangements is a major field of research, resting on the border between algebraic geometry and combinatorics. 
There are two natural objects associated to a hyperplane arrangement.
We will describe both of these constructions in detail in Section~\ref{sec:MatsGrs}.

The first is the \newword{matroid} of the hyperplane arrangement, which can be thought of as encoding the combinatorial structure of the arrangement. 

The second, which captures the geometric structure of the arrangement, is a point in the Grassmannian $G(d,n)$.
There is ambiguity in the choice of this point; it is only determined up to
the action of an $n$-dimensional torus on $G(d,n)$.  
So more precisely, to any hyperplane arrangement, we associate an orbit in $G(d,n)$ for this torus action.
It is technically more convenient to work with the closure of this orbit.
In~\cite{KT1}, the second author suggested that the $K$-class of this orbit could give rise to useful invariants of matroids, thus exploiting the geometric structure to study the combinatorial one.
In this paper, we continue that project.

One of our results is a formula for the Tutte polynomial, 
the most famous of matroid invariants, in terms of the $K$-class of $Y$. 
In addition, we continue the project which was begun in the appendix of~\cite{KT1}, rewriting all of the $K$-theoretic definitions in terms of moment graphs.  
This makes our theory purely combinatorial and in principle completely computable. 
Many results which were shown for realizable matroids in~\cite{KT1} are now extended to all matroids.

We state our two main results.
The necessary $K$-theoretic definitions will be given in the following section.
Given integers $0<d_1<\cdots<d_s<n$, let $\Fl(d_1,\ldots,d_s; n)$ be the partial flag manifold of flags of dimensions $(d_1,\ldots,d_s)$.  For instance, $\Fl(d;n)=G(d,n)$. 
Note that $\Fl(1,n-1;n)$ embeds as a hypersurface in $\PP^{n-1} \times \PP^{n-1}$, regarded as
the space of pairs $(\mathtt{line}, \mathtt{hyperplane})$ in $n$-space.

We will be particularly concerned with the maps in diagram~(\ref{diag1}):
\begin{equation}\label{diag1}\xymatrix{
 & \Fl(1,d,n-1; n)\ar[dl]_{\pi_d}\ar[dr] \ar[ddr]_{\pi_{1(n-1)}} & \\
G(d,n) & & \Fl(1,n-1;n)\ar@{^(->}[d] \\
 & & \PP^{n-1} \times \PP^{n-1}
}\end{equation}
Here the maps $\Fl(1,d,n-1;n) \to \Fl(1,n-1;n)$ and $\Fl(1,d,n-1;n) \to G(d,n)$ are given by respectively forgetting the $d$-plane and forgetting the $1$ and $(n-1)$-planes.
The map $\pi_{1(n-1)}$ is defined by the composition $\Fl(1,d,n-1;n) \to \Fl(1,n-1;n) \to \PP^{n-1} \times \PP^{n-1}$.


Let $T$ be the torus $(\CC^*)^n$, which acts on the spaces in~(\ref{diag1}) in an obvious way. 
Let $x$ be a point of $G(d,n)$, $M$ the corresponding matroid, and $\overline{Tx}$ the closure of the $T$ orbit through $x$. 
Let $Y$ be the class of the structure sheaf of $\overline{Tx}$ in $K^0(G(d,n))$. 
Write $K^0(\PP^{n-1} \times \PP^{n-1}) = \QQ[\alpha, \beta]/( \alpha^n, \beta^n )$, 
where $\alpha$ and $\beta$ are the structure sheaves of hyperplanes.

We can now explain the geometric origin of the Tutte polynomial. 

\begin{theoremtutte}
With the above notations,
$$(\pi_{1 (n-1)})_* \pi_d^* \left(Y \cdot  [\mathcal{O}(1)]  \right) = t_M(\alpha, \beta)$$
where $t_M$ is the Tutte polynomial.
\end{theoremtutte}

The constant term of $t_M$ is zero; this corresponds to the fact that $\pi_{1 (n-1)}$
is not surjective onto $\PP^{n-1}\times\PP^{n-1}$ but, rather, has image lying in \mbox{$\Fl(1,n-1;n)$}. 
The linear term of Tutte, $\beta(M) (\alpha+ \beta)$, corresponds to the fact that the map $\pi_d^{-1}(\overline{Tx}) \to \Fl(1,n-1;n)$ is finite of degree $\beta(M)$.

\begin{theoremh}
Also with the above notations,
$$(\pi_{1 (n-1)})_* \pi_d^* \left( Y \right) = h_M(\alpha+\beta-\alpha \beta)$$
where $h_M$ is the polynomial from~\cite{KT1}.
\end{theoremh}

Our results can be pleasingly presented in terms of $\alpha-1$ and~$\beta-1$.  
For instance, in Theorem~\ref{DiagonalPolynomial},
$h_M$ is a polynomial in $1-(\alpha+\beta-\alpha\beta)=(\alpha-1)(\beta-1)$,
and Theorem~\ref{TutteTheoremProved} obtains
the rank generating function of~$M$
in the variables $\alpha-1$, $\beta-1$.  In other words, 
we might take as a generating set for $K^0(\PP^{n-1}\times\PP^{n-1})$ not the structure sheaves
of linear spaces $\{\alpha^p\beta^q\}$, but the line bundles $\OO(-p,-q)=\OO(-1,0)^p\;\OO(0,-1)^q$.
We have a short exact sequence 
\begin{equation}\label{eq:OOH}
0 \to \O(-1,0) \to \O \to \O_H\to 0,
\end{equation}
where $H$ is any hyperplane $\PP^{n-2} \times \PP^{n-1}\subseteq\PP^{n-1} \times \PP^{n-1}$,
so $[\OO(-1,0)]=1-\alpha$, and similarly $[\OO(0,-1)]=1-\beta$. 

This paper begins by introducing the limited subset of $K$-theory which we need, 
with particular attention to the method of \newword{equivariant localization}.
Many of our proofs, including those of Theorems~\ref{TutteTheoremProved}
and Theorems~\ref{OldNew} above, rely heavily on equivariant $K$-theory,
even though they are theorems about ordinary $K$-theory.
The end of Section~\ref{sec:Kbackground} describes the $K$-theory of the Grassmannian from the equivariant perspective, 
and Section~\ref{sec:MatsGrs} describes the $K$-theory classes associated to matroids.  
We have also written a quick guide to $K$-theory for combinatorialists, in section~\ref{CheatCard}, which focuses on how to do computations rather than on covering precise definitions.
Non-combinatorialists may also find this useful!


Any function on matroids arising from $K^0(G(d,n))$
is a \newword{valuation}.  This is the subject of Section~\ref{sec:valuations},
where we show that the converse doesn't hold by exhibiting a valuative matroid invariant 
not arising from~$K^0(G(d,n))$.

Section~\ref{sec:pushlemma} proves Lemma~\ref{pushtoproj}, the core lemma which we use to push and pull
$K$-classes in diagram~\eqref{diag1}.  In conjunction with equivariant localization,
our computations are reduced to manipulating sums of Hilbert series 
of certain infinite-dimensional $T$-representations, which we may regard as rational functions.
We control these rational functions by expanding them as Laurent
series with various domains of convergence.  We collect a number of
results on this subject in Section~\ref{sec:flipping}.  

Section~\ref{sec:tutte} and Section~\ref{sec:h} are the proofs of the theorems above.  
Finally, Section~\ref{sec:matroidops} takes results from~\cite{KT1}, concerning the behavior of $h_M$ under duality, direct sum and two-sum, and extends them to nonrealizable matroids.

\subsection{Notation}

We write $[n]$ for $\{ 1,2, \ldots, n \}$. For any set $S$, we write $\binom{S}{k}$ for the set of $k$-element subsets of $S$ and $2^S$ for the set of all subsets of $S$.
The use of the notation $I \setminus J$ does not imply that $J$ is contained in $I$.
In addition to the notations $\PP$, $G(d,n)$ and $\Fl$ introduced above, we will write 
$\mathbb{A}^n$ for affine space.

\subsection{Acknowledgments}

The second author was supported by a Research Fellowship from the Clay Mathematics Institute.
We are grateful to David Ben-Zvi, Megumi Harada, Allen Knutson and Sam Payne for providing us with references about and insights into equivariant $K$-theory. 
This paper was finished while the authors visited the American Institute of Mathematics and we are grateful to that institution for the many helpful conversations they fostered.

\section{Background on $K$-theory}\label{sec:Kbackground}

In this section, we will introduce the requisite background on $K$-theory, emphasizing equivariant methods and localization. We have the difficulty of writing for two audiences: combinatorialists who will want to know what $K$-theory is and how to work with it effectively, and algebraic geometers who will want to make sure that we are in fact computing in the $K$-theory which they know and love. We address the second audience first; in section~\ref{CheatCard} we provide a rapid summary aimed at the combinatorial reader.

\subsection{Definition of $K_0$}
If $X$ is any algebraic variety, then $K_0(X)$ denotes the free abelian group generated by isomorphism classes of coherent sheaves on $X$, subject to the relation $[A] + [C] = [B]$ whenever there is a short exact sequence $0 \to A \to B \to C \to 0$. The subspace generated by the classes of vector bundles is denoted $K^0(X)$. If $X$ is smooth, as all the spaces we deal with will be, the inclusion $K^0(X) \into K_0(X)$ is an equality. (See~\cite[Proposition 2.1]{Nie} for this fact, and its equivariant generalization.)

We put a ring structure on $K^0(X)$, generated by the relations $[E] [F] = [E \otimes F]$ for any vector bundles $E$ and $F$ on $X$.
The group $K_0(X)$ is a module for $K^0(X)$, with multiplication given by  $[E] [F] = [E \otimes F]$ where $E$ is a vector bundle and $F$ a coherent sheaf.

For any map $f : X \to Y$, there is a pull back map $f^* : K^0(Y) \to K^0(X)$ given by $f^* [E] = [f^* E]$.
This is a ring homomorphism.
If $f: X \to Y$ is a proper map, there is also a pushforward map $f_*: K_0(X) \to K_0(Y)$ given by 
$$f_* [E] = \sum (-1)^i [R^i f_* E].$$
These two maps are related by the projection formula, which asserts that
\begin{equation}
f_* \big( (f^* [E]) [F] \big) = [E] f_* [F]. \label{ProjFormula} 
\end{equation}
That is, $f_\ast$ is a $K^0(Y)$-module homomorphism, if $K^0(X)$ has the module structure
induced by~$f^\ast$. 

We always have a map from $X$ to a point. 
We denote the pushforward along this map by $\int$, or by $\int_X$ when necessary.\footnote{For the curious reader: There are many analogies between $K^{0}$ and $H^*$. In cohomology, the pushforward from an oriented compact manifold to a point is often denoted by $\int$, because it is given by integration in the deRham formulation of cohomology. We use the same symbol here by analogy.}  
Notice that $K_0(\pt) = K^0(\pt) = \ZZ$, and $\int [E] $ is the holomorphic Euler characteristic of the sheaf $E$.

\subsection{Equivariant $K$-theory}

If $T$ is a torus acting on $X$, then we can form the analogous constructions using $T$-equivariant vector bundles and sheaves.
These are denoted $K^0_T(X)$ and $K_0^T(X)$.
Writing $\Char(T)$ for the lattice of characters, $\mathrm{Hom}(T, \mathbb{C}^*)$, we have $K_0^T(\pt) = K^0_T(\pt) = \ZZ[\Char(T)]$.
Explicitly, a $T$-equivariant sheaf on $\pt$ is simply a vector space with a $T$-action, and the corresponding element of $\ZZ[\Char(T)]$ is the character.

We adopt the abbreviation $\KTP$ for $\ZZ[\Char(T)]$. 
We write $\eqv{[E]}$ for the class of the sheaf $E$ in $K^0_T(X)$.
We also write $\int^T$ for the pushforward to a point in equivariant cohomology.

We pause to discuss Hilbert series and sign conventions.
If $V$ is a finite dimensional representation of $T$, the Hilbert series of $V$ is the sum 
$$\hilb(V) := \sum_{\chi \in \Char(T)} \dim \mathrm{Hom}(\chi, V) \cdot \chi $$
in $\ZZ[\Char(T)]$. 
If $V$ isn't finite dimensional, but $\mathrm{Hom}(\chi, V)$ is for every character $\chi$, then we can still consider this as a formal sum.

Here is one example of particular interest:
let $W$ be a finite dimensional representation of $T$ with character $\sum \chi_i$.
Suppose that all of the $\chi_i$ lie in an open half space in $\Char(T) \otimes \RR$; if this condition holds, we say that $W$ is \newword{contracting}.
Then the Hilbert series of $\Sym(W)$, defined as a formal power series,
represents the rational function $1/(1-\chi_1) \cdots (1-\chi_r)$.
If $M$ is a finitely generated $\Sym(W)$ module, then the Hilbert series of $M$ will 
likewise represent an element of $\mathrm{Frac}(\ZZ[\Char(T)])$ \cite[Theorem 8.20]{MS}.

\begin{remark} 
If $W$ is not contracting, then $\mathrm{Hom}(\chi, \Sym(W)) $ will usually be infinite dimensional.
It is still possible to define Hilbert series in this situation, see~\cite[Section 8.4]{MS}, but we will not need this.
\end{remark}

We now discuss a potentially confusing issue of sign conventions.
Suppose that a group $G$ acts on a ring $A$.
The group $G$ then acts on $\Spec A$ by $g(a) = (g^{-1})^* a$. 
This definition is necessary in order to make sure that both actions are \emph{left} actions.
Although we will only consider actions of abelian groups, for which left and right actions are the same, we still follow this convention.
This means that, if $V$ is a vector space on which $T$ acts by characters $\alpha_1$, $\alpha_2$, \dots, $\alpha_r$, then the coordinate ring of $V$ is $\Sym(V^*)$ and has Hilbert series $1/\prod (1-\alpha_i^{-1})$.
Now, let $W$ be another $T$-representation, with characters $\beta_1$, $\beta_2$, \dots, $\beta_s$.
Consider $W \times V$ as a trivial vector bundle over $V$. 
The corresponding $\Sym(V^*)$ module is $W \otimes \Sym(V^*)$, and has Hilbert series $(\sum \beta_j)/\prod (1-\alpha_i^{-1})$. 
So one cannot simply memorize a rule like ``always invert characters" or ``never invert characters".

When we work out examples, we will need to specify how $T$ acts on various partial flag varieties.
Our convention is that $T$ acts on $\AA^n$ by the characters $t_1^{-1}$, \dots, $t_n^{-1}$. 
Grassmannians, and other partial flag varieties, are flags of \emph{subspaces}, not quotient spaces, and $T$ acts on them by acting on the subobjects of $\AA^n$.
The advantage of this convention is that, for any ample line bundle $L$ on $\Fl(n)$, the action on $\int^T L$ will be by positive powers of the $t_i$.

\begin{example}\label{GrassNbhd}
Let $L$ be the $d$-plane $\Span(e_1, e_2, \ldots, e_d)$ and $M$ be the $(n-d)$-plane $\Span(e_{d+1}, \ldots, e_n)$.
Let $W \subset G(d,n)$ be those linear spaces which can be written as the graph of a linear map $L \to M$. 
This is an open neighborhood of $L$, sometimes called the big Schubert cell.
The cell $W$ is a vector space of dimension $d(n-d)$, naturally identified with $\mathrm{Hom}(L, M)$.
The torus $T$ acts on the vector space $W$ with characters $t_i t_j^{-1}$, for $1 \leq i \leq d$ and $d+1 \leq j \leq n$.
So $T$ acts on the coordinate ring of $W$ with characters $t_i^{-1} t_j$, for $i$ and $j$ as above.
\end{example}

%

\subsection{Localization}

The results in this section are well known to experts, but it seems difficult to find a reference that records them all in one place.
We have attempted to do so; we have made no attempt to find the original sources for these results.
The reader may want to compare our presentation to the description of equivariant cohomology in~\cite{KnutTao}.

In this paper, we will be only concerned with $K^T_0(X)$ for extremely nice spaces $X$.
In fact, the only spaces we will need in the paper are partial flag manifolds and products thereof.
All of these spaces are \newword{equivariantly formal} spaces, meaning that their $K$-theory can be described using the method of \newword{equivariant localization}, which we now explain.

We will gradually add niceness hypotheses on $X$
as we need them.
\begin{Condition} \label{SmoothProj}
Let $X$ be a smooth projective variety with an action of a torus $T$. 
\end{Condition}

Writing $X^T$ for the subvariety of $T$-fixed points, we have a restriction map
$$K^0_T(X) \to K^0_T(X^T) \isom K^0(X^T) \otimes \KTP.$$
Suppose we have:
\begin{Condition} \label{FiniteT0}
$X$ has finitely many $T$-fixed points.
\end{Condition}

\begin{theorem}[{\cite[Theorem 3.2]{Nie}}, see also {\cite[Theorem A.4]{KR}} and {\cite[Corollary 5.11]{VV}}]  \label{thm:localization injects}
In the presence of Condition~\ref{SmoothProj}, the restriction map $K^0_T(X) \to K^0_T(X^T)$ is an injection.
If we have Conditions~\ref{SmoothProj} and~\ref{FiniteT0}, then $K^0_T(X^T)$ is simply the ring of functions from $X^T$ to $\KTP$.
\end{theorem}

For example, if $X = G(d,n)$ and $T$ is the standard $n$-dimensional torus, then $X^T$ is $\binom{n}{d}$ distinct points, one for each $d$-dimensional coordinate plane in 
$\CC^n$. 

Let $x$ be a fixed point of the torus action on $X$,
so we have a restriction map $K^0_T(X) \to K^0_T(x) \isom \KTP$. 
It is important to understand how this map is explicitly computed.
For $\xi\in K^0_T(X)$, we write $\xi(x)$ for the image of $\xi$ in $K^0_T(x)$. 

We adopt a simplifying definition, which will hold in all of our examples: 
We say that $X$ is \newword{contracting at $x$} if there is a $T$-equivariant neighborhood of~$x$ 
which is isomorphic to $\AA^N$ with $T$ acting by a contracting linear representation.
We will call the action of $T$ on $X$ \newword{contracting} if it is contracting at every $T$-fixed point.

Let $x$ be contracting. 
Let $U$ be a $T$-equivariant neighborhood of~$x$ isomorphic to a contracting $T$-representation, 
and let $\chi_1$, \dots, $\chi_N$ be the characters by which $T$ acts on $U$.
Let $E$ be a $T$-equivariant coherent sheaf on $U$, corresponding to a graded, finitely generated $\OO(U)$-module $M$.
Then the Hilbert series of $M$ lies in $\mathrm{Frac}(\ZZ[\Char(T)])$;  it is a rational function of the form $k(E)/\prod(1-\chi_i^{-1})$ for some polynomial $k(E)$  in $\ZZ[\Char(T)]$. 

\begin{theorem}\label{thm:Hilb}
If $U$ is an open neighborhood of $x$ as above then $K^0_T(U) \isom \KTP$.
With the above notations, $\eqv{[E]}(x) = k(E)$.
\end{theorem}

\begin{proof}[Proof sketch]
The restriction map $K^0_T(X) \to K^0_T(x)$ factors through $K^0_T(U)$, so it is enough to show that $\eqv{[E]}|_U$ is $k(E)$.

Let $M$ be the $\mathcal{O}(U)$-module coresponding to $U$, and abbreviate $\mathcal{O}(U)$ to $S$.
Then $M$ has a finite $T$-graded resolution by free $S$-modules as in \cite[Chapter 8]{MS}, say\footnote{Because we write our grading group multiplicatively, we write $S[\chi^{-1}]$ where $S[-\chi]$ might appear more familiar. This notation will only arise within this proof.}
$$ 0 \to \bigoplus_{i=1}^{b_N} S[\chi_{iN}^{-1}] \to \cdots \to \bigoplus_{i=1}^{b_1} S[\chi_{i1}^{-1}] \to \bigoplus_{i=1}^{b_0} S[\chi_{i0}^{-1}] \to M \to 0.$$

The sheafification of $S[\chi^{-1}]$, by definition, has class $\chi$ in $K^0_T(U)$.
So 
\begin{equation}
\eqv{[E]} = \sum_{j=1}^N (-1)^i \sum_{i=1}^{b_j} \chi_{i j}. \label{Resolution}
\end{equation}
As the reader can easily check, or read in \cite[Proposition 8.23]{MS}, the sum in~(\ref{Resolution}) is $k(E)$.
\end{proof}

In particular, if $E$ is a vector bundle on $U$, and $T$ acts on the fiber over $x$ with character $\sum \eta_i$, then $\eqv{[E]}(x) =\sum \eta_i$.

\begin{remark}
The positivity assumption is needed only for convenience. In general, let $x$ be a smooth variety with $T$-action, $x$ a fixed point of $X$, and let $E$ be an equivariant coherent sheaf on $X$.
Then $\OO_x$ is a regular local ring, and $E_x$ a finitely generated $\OO_x$ module. Passing to the associated graded ring and module, $\mathrm{gr} \ E_x$ is a finitely generated, $T$-equivariant $(\mathrm{gr} \ \OO_x)$-module, and $\mathrm{gr} \ \OO_x$ is a polynomial ring.
If the $T$-action on the tangent space at $x$ is contracting, then we can define $\eqv{[E]}(x)$ using the 
Hilbert series of $\mathrm{gr} \ E_x$; if not, we can use the trick of~\cite[Section 8.4]{MS} to define $k(\mathrm{gr} \ E_x)$ and, hence, $\eqv{[E]}(x)$. 
But we will not need either of these ideas.
\end{remark}

We have now described, given a $T$-equivariant sheaf $E$ in $K_0^T(X)$, how to describe it as a function from $X^T$ to $\KTP$. 
It will also be worthwhile to know, given a function from $X^T$ to $\KTP$, when it is in $K^0_T(X)$. 
For this, we need 
\begin{Condition} \label{FiniteT1}
There are finitely many $1$-dimensional $T$-orbits in $X$, each of which has closure isomorphic to $\PP^1$. 
\end{Condition}

Each $\mathbb{P}^1$ must contain two $T$-fixed points.

\begin{theorem}[{\cite[Corollary 5.12]{VV}}, see also {\cite[Corollary A.5]{KR}}]  \label{thm:K0T}
Assume conditions~\ref{SmoothProj}, \ref{FiniteT0} and \ref{FiniteT1}.
Let $f$ be a function from $X^T$ to~$\KTP$. 
Then $f$ is of the form $\xi({\cdot})$ for some $\xi\in K^0_T(X)$ if and only if 
the following condition holds: For every one dimensional orbit, 
on which $T$ acts by character $\chi$ and for which $x$ and $y$ are the $T$-fixed points
in the closure of the orbit, we have
$$f(x) \equiv f(y) \mod 1-\chi.$$
\end{theorem}

We cannot conclude that $\xi$ is itself the class $\eqv{[E]}$ of a $T$-equivariant sheaf $E$,
for reasons of positivity. For example, $\xi(x)=-1$ does not describe the class of a sheaf.

\begin{example} \label{GrassEx}
Let's see what this theorem means for the Grassmannian $G(d,n)$.
Here $\KTP$ is the ring of Laurent polynomials $\ZZ[t_1^{\pm}, t_2^{\pm}, \ldots, t_n^{\pm}]$. 
The fixed points $G(d,n)^T$ are the linear spaces of the form $\Span(e_i)_{i \in I}$ for $I \in \binom{[n]}{d}$. We will write this point as $x_I$ for $I \in \binom{[n]}{d}$.
So an element of $K^0_T(G(d,n))$ is a function $f:\binom{[n]}{d} \to \KTP$ obeying certain conditions.
What are those conditions? Each one-dimensional torus orbit joins $x_I$ to $x_J$ where $I = S \sqcup \{ i \}$ and $J = S \sqcup \{ j \}$ for some $S$ in $\binom{[n]}{d-1}$. 
Thus an element of $K^0_T(G(d,n))$ is a function $f:\binom{[n]}{d} \to \KTP$ such that
$$f(S \sqcup \{ i \}) \equiv f(S \sqcup \{ j \}) \mod 1-t_i/t_j$$
for all $S \in \binom{[n]}{d-1}$ and $i$, $j \in [n] \setminus S$.
\end{example}

We now describe how to compute tensor products, pushforwards and pullbacks in the localization description.
The first two are simple. 
Tensor product corresponds to multiplication. That is to say,
\begin{equation}
\left( \eqv{[E]}\eqv{[F]} \right) (x) = \eqv{[E]}(x) \cdot \eqv{[F]}(x). \label{Tensor} 
\end{equation}
Pullback corresponds to pullback. That is to say, if $X$ and $Y$ are equivariantly formal spaces, and $\pi: X \to Y$ a $T$-equivariant map, then
\begin{equation}
\left( \pi^* \eqv{[E]} \right) (x) = \eqv{[E]} \left( \pi(x) \right) \label{Pullback}
\end{equation}
for $x \in X^T$ and $\eqv{[E]} \in K^T_0(X)$. 
The proofs are simply to note that pullback to $X^T$ and $Y^T$ is compatible with pullback and with multiplication in the appropriate ways.

The formula for pushforward is somewhat more complex. 
Let $X$ and $Y$ be pointed and $\pi:X \to Y$ a $T$-equivariant map.
For $x \in X^T$, let $\chi_1(x)$, $\chi_2(x)$, \dots, $\chi_r(x)$ be the characters of $T$ acting on a neighborhood of $x$; for $y \in Y^T$, define $\eta_1(y)$, \dots, $\eta_s(y)$ similarly.
Then we have the formula
\begin{equation}
\frac{(\pi_* \eqv{[E]})(y)}{(1- \eta_1^{-1}(y)) \cdots (1-\eta_s^{-1}(y))} = \sum_{x \in X^T, \ \pi(x) =y} \frac{\eqv{[E]}(x)}{(1-\chi_1^{-1}(x)) \cdots (1-\chi_r^{-1}(x))}. \label{PushForward}
\end{equation}
See \cite[Theorem 5.11.7]{CG}.

It is often more convenient to state this equation in terms of multi-graded Hilbert series. If $\hilb(E_x)$ is the multi-graded Hilbert series of the stalk $E_x$, then equation~(\ref{PushForward}) reads:
\begin{equation}
 \hilb(\pi_*(E)_y) = \sum_{x \in X^T, \ \pi(x) =y} \hilb(E_x) \label{PushForwardHilb}
\end{equation}

It is also important to note how this formula simplifies in the case of the pushforward to a point. In that case, we get
\begin{equation}
 \int_X^T \eqv{[E]} = \sum_{x \in X^T} \hilb(E_x) \label{PushToPoint}
\end{equation}
This special case is more prominent in the literature than the general result~(\ref{PushForward}); see for example \cite[Section 4]{Nie} for some classical applications.

Finally, we describe the relation between ordinary and $T$-equivariant $K$-theories. 
There is a map from equivariant $K$-theory to ordinary $K$-theory by forgetting the $T$-action.
In particular, the map $\KTP \to K^0(\pt) = \ZZ$ just sends every character of $T$ to $1$.
In this way, $\ZZ$ becomes a $\KTP$-module.
Thus, for any space $X$ with a $T$-action, we get a map $K^0_T(X) \otimes_{\KTP} \ZZ \to K^0(X)$.
All we will need is that this map exists, but the reader might be interested to know the stronger result: 
\begin{theorem}[{\cite[Theorem 4.3]{Mer}}]\label{thm:K0TtensorQ}
Assuming Condition~\ref{SmoothProj},  the map $$K^0_T(X) \otimes_{\KTP} \ZZ \to K^0(X)$$ is an isomorphism.
\end{theorem}

\subsection{Moment graphs: a starting point for combinatorialists} \label{CheatCard}

There is a technology known as \newword{moment graphs} which is extremely useful for describing equivariant
functors such as $K$-theory in the situation where localization applies. We begin by describing this in a purely combinatorial setting. We recommend~\cite{GZ} as a reference for the use of moment graphs in equivariant cohomology, which is extremely similar to the $K$-theory setup.

Let $T$ be a torus with character group $\Char(T)$. For the next several paragraphs, the torus $T$ will only appear as a formal symbol, and the reader can think of the free abelian group $\Char(T)$
as the primary object. The ring $K^0_T$ is the group ring $\mathbb{Z}[\Char(T)]$, 
a Laurent polynomial ring. 

Let $\Gamma$ be a finite graph whose vertices are labelled with elements of the lattice $\Char(T)$. For any edge $(v,w)$ of $\Gamma$, let $d(v,w)$ be the minimal lattice vector along $v-w$. Let $(K^0_T)^{\mathrm{Vert}(\Gamma)}$ be the ring of functions $v \mapsto f_v$ from the vertices of $\Gamma$ to $K^0_T$, with componentwise sum and product. Let $K^0_T(\Gamma)$ be the subring of  $(K^0_T)^{\mathrm{Vert}(\Gamma)}$ consisting of those functions such that $f_v \equiv f_w \mod d(v,w)$. 

The reason we have introduced the ring $K^0_T(\Gamma)$ is that many of the rings we care about can be described by this construction, 
in such a way that the vertices of $\Gamma$ correspond to torus fixed points $X^T$. We begin with the case of the Grassmannian. Let $T$ be the $n$-dimensional torus of diagonal matrices in $\GL_n$. Let $G(d,n)$ be the 
Grassmannian of $d$-planes in $n$-space. For $I$ in $\binom{[n]}{d}$, let $e_I \in \ZZ^n$ be the $(0,1)$-vector whose 1 coordinates are those in $I$. Let $\Gamma(d,n)$ be the graph whose vertices are $\{ e_I : I \in \binom{[n]}{d} \}$, and where there is an edge between $e_I$ and $e_J$ if $I$ and $J$ only differ by a single element. Then $K^0_T(G(d,n)) \cong K^0_T(\Gamma(d,n))$. 
We say $\Gamma(d,n)$ is the moment graph of $G(d,n)$.

To be more explicit, $K^0_T(G(d,n))$ is the ring of maps $I \mapsto f_I$ from $\binom{[n]}{d}$ to  $\ZZ[t_1^{\pm}, \ldots, t_n^{\pm}]$ such that $f_{Si} \equiv f_{Sj} \mod t_i-t_j$.  See example~\ref{GrassEx}. The group $S_n$ acts on $K^0_T(G(d,n))$ both by permuting the vertices of $\Gamma(d,n)$ and by acting on the Laurent polynomial ring $K^0_T \cong \ZZ[t_1^{\pm}, \ldots, t_n^{\pm}]$. 

More generally, we will want to think about partial flag varieties. 
Let $\Fl(d_1, d_2, \ldots, d_r; n)$ be the partial flag variety of flags of dimensions 
$d_1 <  d_2 <  \ldots < d_r$. Let $I_{\bullet} = (I_1, I_2, \ldots, I_r)$ be a chain of subsets of $[n]$, with $|I_j|=d_j$ and $I_1 \subset I_2 \subset \cdots \subset I_r$. Let $e(I_{\bullet}) = \sum e_{I_j}$. Let $\Gamma(d_1, \ldots, d_r; n)$ be the graph whose vertices are the $e(I_{\bullet})$, as $I_{\bullet}$ ranges over chains of subsets as above, and where there is an edge between $(I_1, I_2, \ldots, I_r)$ and $(J_1, J_2, \ldots, J_r)$ if (a) for every index $k$ except one, we have $I_k=J_k$ and (b) for that one index, $I_k$ and $J_k$ differ by a single element. Then $K^0_T(\Fl(d_1, d_2, \ldots, d_r; n)) = K^0_T(\Gamma(d_1, \ldots, d_r; n))$.

We will also work with products of flag varieties.   
We have $K^0_{T \times T'}(X \times X')  \cong  K^0_T(X) \otimes K^0_{T'}(X')$.
This can be done in moment graphs also,
$K^0_{T\times T'}(\Gamma\times\Gamma')\cong K^0_T(\Gamma) \otimes K^0_{T'}(\Gamma')$,
where $\Gamma\times\Gamma'$ is the Cartesian product graph.

And we will sometimes need to work with tori smaller than the full group of diagonal matrices. 
If we have an embedding of tori $S \subset T$, then we get a surjection of character groups $\Char(T) \to \Char(S)$, and hence a surjection of rings $K^0_T \to K^0_S$. 
For any homogeneous space $X$ (and in greater generality, which we don't need) we have 
$$K^0_S(X) \cong K^0_T(X) \otimes_{K^0_T} K^0_S.$$
For $X$ a homogeneous space $G/P$, and $T$ any subtorus of $G$, the ring $K^0_T(X)$ is a free finite rank $K^0_T$-module, so this tensor product is easy to describe as an abelian group (though potentially quite subtle as a ring). 
The rank of this module is the number of $T$ fixed points, namely
$|W/W_P|$, where $W$ is the corresponding Coxeter group and $W_P$ the parabolic subgroup. For example, in the case of $G(d,n)$, this module has rank $\binom{n}{d}$. 

A particularly important case of the previous paragraph is when $S$ is the trivial torus, so $K^0_S \cong \ZZ$. In this case, we drop the subscript $S$. Note, in particular, that the $S_n$ action becomes trivial 
on the non-equivariant $K$-theory $K^0(G(d,n))$.
We reinforce that passing through equivariant $K$-theory is essential to obtain a description
of non-equivariant $K$-theory using moment graphs; there is no natural way to interpret
tensoring from $K^0_T$ down to~$\ZZ$ as modifying our definition of $K^0_T(\Gamma)$ to involve no characters.

We can define the ring $K^0_T(X)$ whenever $X$ is a variety equipped with a $T$-action. 
Classes in this ring come from sheaves on $X$. 
There are two important sources of such sheaves: subvarieties of $X$, and vector bundles on $X$. 
In the case where $X$ is a homogeneous space $G/P$, we have that $K^0_T(X)$ is a free $K^0_T$-module on the classes of the Schubert sheaves. 
This should remind the reader of Schubert calculus;  we refer the reader to~\cite{Buch} for a guide to $K^0(G(d,n))$ which pursues this analogy. 
See~\cite{Knutson} for a description of the class in $K^0_T(\Fl(n))$ corresponding to a Schubert variety.  The corresponding formula for $G(d,n)$ should be extractable from~\cite{Knutson}, but does not appear to have been published.
For us, the most important subvarieties of $G(d,n)$ will be not Schubert varieties, but torus orbit closures. 
There is one of these for each matroid realizable over $\mathbb{C}$; we discuss how to assign a class in $K^0_T(G(d,n))$ to a matroid in section~\ref{sec:MatsGrs}.

The other important classes in $K^0_T(X)$ are those coming from vector bundles. 
If $S$ is the tautological $d$-dimensional bundle over $G(d,n)$, and $\mathbb S_{\lambda}$ is a Schur functor, consider the vector bundle $\mathbb{S}_{\lambda}(S)$. The corresponding class in $K^0_T(G(d,n))$ assigns, to the vertex $e_I$, the Schur polynomial $s_{\lambda}(t_i^{-1})_{i \in I}$. If we use 
the dual bundle $S^{\vee}$, the quotient bundle $Q$ or the dual quotient bundle $Q^{\vee}$ instead, then we get the Schur polynomials $s_{\lambda}(t_i)_{i \in I}$, $s_{\lambda}(t_j)_{j \not \in I}$ and $s_{\lambda}(t_j^{-1})_{j \not \in I}$ respectively.

Finally, we discuss the functorial properties of $K^0_T$. Given a map $f_*: X \to Y$, equivariant with respect to a $T$-action, we get both a pull-back $f^*: K^0_T(Y) \to K^0_T(X)$ and (given a condition called properness, which is true in our examples) a push-forward $f^*: K^0_T(X) \to K^0_T(Y)$. These are computed in the moment graph setting by equations~\eqref{Pullback} and~\eqref{PushForward}.

\section{Matroids and Grassmannians}\label{sec:MatsGrs}

Let $\Elts$ be a finite set (the \newword{ground set}), which we will usually take to be $[n]$. 
For $I \subseteq \Elts$, we write $e_I$ for the vector $\sum_{i \in I} e_i$ in $\ZZ^{\Elts}$. 

Let $M$ be a collection of $d$-element subsets of $\Elts$. 
Let $\Poly(M)$ be the convex hull of the vectors $e_I$, as $I$ runs through $M$.
The collection $M$ is called a matroid if it obeys any of a number of equivalent conditions.
Our favorite is due to Edmonds:
\begin{theorem}[\cite{Edmonds}; see also {\cite[Theorem 4.1]{GGMS}}]
 $M$ is a matroid if and only if $M$ is nonempty and every edge of $\Poly(M)$ is in the direction $e_i-e_j$ for some $i$ and $j \in \Elts$. 
\end{theorem}
See \cite{Crapo} for motivation and \cite{NW} for more standard definitions.

We now explain the connection between matroids and Grassmannians.
We assume basic familiarity with Grassmannians and their Pl\"ucker embedding.  See \cite[Chapter 14]{MS} for background.
Given a point $x$ in $G(d,n)$, the set of $I$ for which the Pl\"ucker coordinate $p_I(x)$ is nonzero forms a matroid, which we denote $\Mat(x)$.  (A matroid of this form is called \newword{realizable}.)
Let $T$ be the torus $(\CC^*)^n$, which acts on $G(d,n)$ in the obvious way,
so that $p_I(tx)= t^{e_I} p_I(x)$ for $t\in T$.
Clearly, $\Mat(tx)= \Mat(x)$ for any $t \in T$.

\begin{remark} We pause to explain the connection to hyperplane arrangements, although this will only be needed for motivation.
Let $H_1$, $H_2$, \dots, $H_n$ be a collection of hyperplanes through the origin in $\CC^d$.
Let $v_i$ be a normal vector to $H_i$.
Then the row span of the $d \times n$ matrix $\left( v_1 \ v_2 \ \cdots \ v_n \right)$ is a point in $G(d,n)$. 
This point is determined by the hyperplane arrangement, up to the action of $T$.
Thus, it is reasonable to study hyperplane arrangements by studying $T$-invariant properties of $x$.
In particular, $\Mat(x)$ is an invariant of the hyperplane arrangement. 
It follows easily from the definitions that $\{ i_1, \ldots, i_d \}$ is in $\Mat(x)$ if and only if the hyperplanes $H_{i_1}$, \dots, $H_{i_d}$ are transverse.
\end{remark}

We now discuss how we will bring $K$-theory into the picture.
Consider the torus orbit closure $\overline{Tx}$. 
The orbit $Tx$ is a translate (by $x$) of the image of the monomial map
given by the set of characters $\{t^{-e_I} : p_I(x)\neq0\}$. Essentially\footnote{We say essentially for two reasons. First, Cox describes the toric variety associated to a polytope $P$ as a the Zariski closure of the image of $t \mapsto (t^p)_{p \in P \cap \ZZ^n}$. We would rather describe it as the Zariski clsoure of the image of $t \mapsto (t^{-p})_{p \in P \cap \ZZ^n}$. 
These are the same subvariety of $G(d,n)$, and the same class in $K$-theory, but our convention makes the obvious torus action on the toric variety match the restriction of the torus action on $G(d,n)$. The reader may wish to check that our conventions are compatible with Example~\ref{GrassNbhd}. 

Second, there is a potential issue regarding normality here. According to most references, the toric variety associated to $\Poly(\Mat(x))$ is the normalization of $\overline{Tx}$. See the discussion in~\cite[Section 5]{Cox}. However, this issue does not arise for us because $\overline{Tx}$ is normal and, in fact, projectively normal; see~\cite{White}.
}
 by definition, $\overline{Tx}$ is the toric variety associated to the polytope $\Poly(\Mat(x))$ (see \cite[Section 5]{Cox}).
In the appendix to~\cite{KT1}, the second author checked that the class of 
the structure sheaf of $\overline{Tx}$ in~$K^0_T(G(d,n))$ depends only on $\Mat(x)$, 
and gave a natural way to define a class $y(M)$ in $K^0_T(G(d,n))$ for any matroid $M$ of rank $d$ on $[n]$,
nonrealizable matroids included.

We review this construction here. 
For a polyhedron $P$ and a point $v\in P$,
define $\Cone_v(P)$ to be the positive real span of all vectors of the form $u-v$, with $u \in P$;
if $v$ is not in~$P$, define $\Cone_v(P)=\emptyset$.  
Let $M \subseteq \binom{[n]}{d}$ be a matroid. 
We will abbreviate $\Cone_{e_I}(\Poly(M))$ by $\Cone_I(M)$.
For a pointed rational polyhedron $C$ in~$\RR^n$, define $\hilb(C)$ to be the Hilbert series
$$\hilb(C) := \sum_{a\in C\cap\ZZ^n} t^a.$$
This is a rational function with denominator dividing $\prod_{i \in I} \prod_{j \not \in I} (1- t_i^{-1} t_j)$ \cite[Theorem 4.6.11]{Stanley}.
We define the class $y(M)$ in $K^0_T(G(d,n))$ by 
$$y(M)(x_I) := 
\hilb(\Cone_I(M)) \prod_{i \in I} \prod_{j \not \in I} (1-t_i^{-1} t_j),$$ 
Note that $\hilb(\Cone_I(M))=0$ for $I \not \in M$. 

To motivate this definition, suppose $M$ is of the form $\Mat(x)$ for some $x \in G(d,n)$. 
For $I$ in $M$, the toric variety $\overline{T x}$ is isomorphic near $x_I$ to $\Spec \CC[\Cone_I(M)\cap\ZZ^n]$. 
In particular, the Hilbert series of the structure sheaf of $\overline{Tx}$ near $x_I$ is $\hilb(\Cone_I(M))$.
So in this situation $y(M)$ is exactly the $T$-equivariant class of the structure sheaf of~$\overline{Tx}$.

We now prove the following fact, which was stated without proof in~\cite{KT1} as Proposition~A.6.
\begin{prop} \label{YWellDefined}
 Whether or not $M$ is realizable, the function $y(M)$ from $G(d,n)^T$ to $\KTP$ defines a class in $K^0_T(G(d,n))$.
\end{prop}

This follows from a more general polyhedral result.

\begin{lemma} \label{EdgeCancellation}
Let $P$ be a lattice polytope in $\RR^n$ and let $u$ and $v$ be vertices of $P$ connected by an edge of $P$. Let $e$ be the minimal lattice vector along the edge pointing from $u$ to $v$, with $v=u+ke$. Then $\hilb(\Cone_u(P))+\hilb(\Cone_v(P))$ is a rational function whose denominator is \textbf{not} divisible by $1-t^e$.
\end{lemma}

It is not too hard to give a direct proof of this result, but we take a shortcut and use Brion's formula.

\begin{proof}
Note that the truth of the claim is preserved under dilating the polytope by some positive integer $N$, since this does not effect the cones at the vertices.

Since $u$ and $v$ are joined by an edge, we can find a hyperplane $H$ such that $u$ and $v$ lie on one side of $H$, and the other vertices of $P$ lie on the other. 
Perturbing $H$, we may assume that it is not parallel to $e$, and that the defining equation of $H$ has rational coefficients.
Let $H^+$ be the closed half space bounded by $H$, containing $u$ and $v$.
Then $H^+ \cap P$ is a bounded polytope and, after dilation, we may assume that it is a lattice polytope.

By Brion's formula~(\cite{BHS}, \cite{Bri}) applied to give the Ehrhart polynomial at $0$, 
$$\sum_{v \in \mathrm{Vert}(P \cap H^+)}  \hilb(\Cone_w(H^{+} \cap P)) = 1.$$
The terms coming from vertices $w$ other than $u$ and $v$ have denominators not divisible by $(1-t^e)$.
So $\hilb(\Cone_u(P \cap H^{+}))+\hilb(\Cone_v(P \cap H^{+}))$, which is $\hilb(\Cone_u(P))+\hilb(\Cone_v(P))$, also has denominator not divisible by $1-t^e$.
\end{proof}

\begin{proof}[Proof of Proposition~\ref{YWellDefined}]
We must check the conditions of Theorem~\ref{thm:K0T}.
If $x_I$ and $x_J$ are two fixed points of $G(d,n)$, joined by a one dimensional orbit, then we must have $I=S \cup \{ i \}$ and $J = S \cup \{ j \}$ for some $S \in \binom{[n]}{k-1}$ with $i$, $j \in [n] \setminus S$. We must check that $y(M)(x_I) \equiv y(M)(x_J) \mod 1-t_i^{-1} t_j$.
Abbreviate $\prod_{a \in I} \prod_{b \in [n] \setminus I} (1-t_a^{-1} t_b)$ to $d_I$, and define $d_J$ similarly.
Observe that $d_I \equiv d_J \equiv 0 \mod 1-t_i^{-1} t_j$ and $d_I \equiv - d_J \mod (1-t_i t_j^{-1})^2$.

If $I$ and $J$ are not in $M$, then $y(M)(x_I) = y(M)(x_J) = 0$.

Suppose that  $I \in M$ and $J \not \in M$.
Since $\hilb(\Cone_I(M))$ has no edge in direction $e_i-e_j$, the denominator of $\hilb(\Cone_I(M))$ is not divisible by $1-t_i^{-1} t_j$.
So $y(M)(x_I) = d_I \hilb(\Cone_I(M))$ is divisible by $1- t_i^{-1} t_j$, as required.

If $I$ and $J$ are in $M$, then we apply Lemma~\ref{EdgeCancellation} to see that the denominator of $\hilb(\Cone_I(M)) + \hilb(\Cone_J(M))$
is not divisible by $1-t_i^{-1} t_j$. 
Also, the denominator of $\hilb(\Cone_J(M))$ is only divisble by $1-t_i^{-1} t_j$ once.

Writing
\begin{multline*}
y(M)_I - y(M)_J = d_I \hilb(\Cone_I(M)) - d_J \hilb(\Cone_J(M)) =  \\ d_I \left( \hilb(\Cone_I(M)) + \hilb(\Cone_J(M)) \right) - (d_I + d_J)  \hilb(\Cone_J(M)), 
\end{multline*}
we see that each term on the righthand side is divisible by $1-t_i^{-1} t_j$.
\end{proof}

Although we will not need this fact, it follows from the Basis Exchange theorem \cite[Lemma 1.2.2]{Oxley}
that the semigroup $\Cone_I(M)\cap\ZZ^n$ 
is generated by the first nonzero lattice point on each edge of~$\Cone_I(M)$, i.e.\ by
the vectors $e_j-e_i$, where $(i,j)$ ranges over the pairs $i \in I$, $j \not \in I$ such that $I \cup \{ j \} \setminus \{ i \}$ is in $M$.

\begin{example}\label{SquarePyramid1}
We work through these definitions for the case of a matroid in~$G(2,4)$, namely 
$$M=\{13,14,23,24,34\}.$$
This $M$ is realizable, arising as $\Mat(x)$ when for instance $x$ is the row\-span of
$\begin{pmatrix}1 & 1 & 0 & 1 \\ 0 & 0 & 1 & 1\end{pmatrix}$,
with Pl\"ucker coordinates $(p_{12},p_{13},p_{14},p_{23},p_{24},p_{34})=(0,1,1,1,1,-1)$.
For $t\in T$ the Pl\"ucker coordinates of $tx$ are 
$$(0,t_1t_3,t_1t_4,t_2t_3,t_2t_4,-t_3t_4).$$
Every point of~$G(2,4)$ with $p_{12}=0$ and the other Pl\"ucker coordinates nonzero
can be written in this form, so $\overline{Tx}$ has defining equation $p_{12}=0$ in $G(2,4)$.
(That is, $\overline{Tx}$ is the Schubert subvariety $\Omega_{13}$ of which $x$ is an interior point.  
Compare Remark~\ref{y isn't Schubert}.)

Computing $y(M)$ entails finding the Hilbert functions
$\hilb(\Cone_I(M))$ for each~$I\in M$.  
The cone $\Cone_{13}(M)$ is a unimodular simplicial cone with ray generators
$e_2-e_1$, $e_4-e_1$, and $e_4-e_3$, so we have 
$$\hilb(\Cone_{13}(M))=\frac{1}{(1-t_1^{-1} t_2)(1-t_1^{-1} t_4)(1-t_3^{-1} t_4)}.$$
We can do similarly for the other bases $14$, $23$ and $24$.
At $I=34$, the cone
$\Cone_{34}(M)$ is the cone over a square with ray directions 
$e_1-e_3$, $e_1-e_4$, $e_2-e_3$, and $e_2-e_4$.  By summing over a triangulation
of this cone we find that
$$\hilb(\Cone_{34}(M))=\frac{1-t_1t_2 t_3^{-1} t_4^{-1}}{(1-t_1 t_3^{-1})(1-t_1 t_4^{-1})(1-t_2 t_3^{-1})(1-t_2 t_4^{-1})}.$$
Accordingly, $y(M)$ is sent under the localization map of Theorem~\ref{thm:localization injects} to
$$(0, 1-t_2 t_3^{-1}, 1-t_2 t_4^{-1}, 1-t_1 t_3^{-1}, 1-t_1 t_4^{-1}, 1-t_1t_2 t_3^{-1} t_4^{-1})$$
again ordering the coordinates lexicographically.
We see that this satisfies the congruences in Theorem~\ref{thm:K0T}.
\end{example}

\section{Valuations}\label{sec:valuations}

A subdivision of a polyhedron $P$ is a polyhedral complex 
$\mathcal D$ with $|\mathcal D|=P$.  We use the names $P_1,\ldots,P_k$
for the facets of a typical subdivision $\mathcal D$ of~$P$, and
for $J\subseteq[k]$ nonempty we write $P_J = \bigcap_{j\in J} P_j$, which is a face of~$\mathcal D$. 
We also put $P_\emptyset=P$.
Let $\mathcal P$ be a set of polyhedra in a vector space $V$, and $A$ an abelian group.
We say that a function $f:\mathcal P\to A$ is a \newword{valuation} (or is \newword{valuative})
if, for any subdivision such that $P_J\in\mathcal P$ for all $J\subseteq[k]$,
we have 
$$\sum_{J\subseteq[k]} (-1)^{|J|} f(P_J) = 0.$$
For example, one valuation of fundamental importance to the theory is the
function $\One(\cdot)$ mapping each polytope $P$ to its characteristic function. 
Namely, $\One(P)$ is the function $V\to\mathbb Z$ which takes the value $1$ on $P$ and $0$ on $V \setminus P$.  

We will be concerned in this paper with the case $\mathcal P=\{\Poly(M):\mbox{$M$ a matroid}\}$,
and we will identify functions on $\mathcal P$ with the corresponding functions
on matroids themselves. 
Many important functions of matroids, including the Tutte polynomial, are valuations.  

We now summarize the results of~\cite{DF}.
A function of matroid polytopes is a valuation if and only if
it factors through $\One$.  Therefore, the group of matroid polytope 
valuations valued in~$A$ is $\Hom(\IndMat,A)$, where
$\IndMat$ is the $\ZZ$-module of functions $V\to\ZZ$
generated by indicator functions of matroid polytopes.
We are also interested in valuative matroid invariants, those valuations which
take equal values on isomorphic matroids.  
For $M$ a matroid on the ground set $\Elts$ and $\sigma\in S_{\Elts}$ a permutation, 
let $\sigma\cdot M$ be the matroid $\{\{\sigma(i_1),\ldots,\sigma(i_d)\} : \{i_1,\ldots,i_d\}\in \Elts\}$.
This action of~$S_n$ induces an action of $S_n$ on $\IndMat$.
We write $\IndMat/S_n$ for the quotient of $\IndMat$ by the subgroup generated by elements of the form $\sigma(M) - M$, with $\sigma \in S_n$ and $M \in \IndMat$. 
The group of valuative invariants 
valued in $A$ is $\Hom(\IndMat,A)^{S_n} = \Hom(\IndMat/S_n,A)$.

Given $I=\{i_1,\ldots,i_d\}\in\binom{[n]}d$ with $i_1<\cdots<i_d$, 
the \newword{Schubert matroid} $\SM(I)$ 
is the matroid consisting of all sets $\{j_1,\ldots,j_d\}\in\binom{[n]}d$,
$j_1<\ldots<j_d$ such that $j_k\geq i_k$ for each~$k\in[d]$.  

Theorem~\ref{thm:DF}, which was Theorems 5.4 and~6.3 of~\cite{DF},
provides explicit bases for $\IndMat$ and $\IndMat/{S_n}$.  The dual bases 
are bases for the groups of valuative matroid functions and invariants, respectively.
\begin{theorem}\label{thm:DF}
For $I\in\binom{[n]}d$, let $\rho(I)\subseteq S_n$ consist of one
representative of each coset of the stabilizer $(S_n)_{\SM(I)}$.
\begin{enumerate}
\renewcommand{\labelenumi}{{\rm (\alph{enumi})}}
\item The set $\{\Poly(\sigma\cdot\SM(I)) : I\in\binom{[n]}d, \sigma\in\rho(I)\}$
is a basis for~$\IndMat$.
\item The set $\{\Poly(\SM(I)) : I\in\binom{[n]}d\}$
is a basis for~$\IndMat/S_n$.
\end{enumerate}
\end{theorem}

\begin{remark}\label{y isn't Schubert}
We caution the reader that 
$y(\SM(I))$ is \emph{not} in general the class of the structure sheaf of the Schubert variety $\Omega_I$.
They differ in that $\Omega_I$ is the closure of the set of all points 
$x \in G(d,n)$ with $\Mat(x) = \SM(I)$, while
$y(M)$ is the class of the closure of the torus orbit through a single point $x$ with  $\Mat(x) = \SM(I)$.
Once $\Omega_I$ is large enough to have multiple torus orbits in its interior, there appears to be no relation between $y(\SM(I))$ and $\eqv{[\OO_{\Omega_I}]}$.
\end{remark}

We now discuss how valuations arise from $K$-theory.
Let $\mathcal D$ be a subdivision of matroid polytopes, with facets $P_1,\ldots,P_k$,
and let $P_J=\Poly(M_J)$.  Then we have a linear relation of $K$-theory classes
\begin{equation}\label{eq:lin rel of K-classes}
\sum_{J\subseteq[k]} (-1)^{|J|}y(M_J)=0.
\end{equation}
That is,
\begin{prop}\label{prop:y is a valuation}
The function $y$ is a valuation of matroids.
\end{prop}

\begin{proof}
Let $I \in \binom{[n]}d$.  We will check that $\sum_{J \subseteq [k]} (-1)^{|J|} y(M_J)(x_I)=0$.  
The nonempty cones among
the $\Cone_I(M_j)$, $j=1,\ldots,k$, are the facets of a polyhedral subdivision,
and $\Cone_I(M_J) = \bigcap_{j\in J}\Cone_I(M_j)$.  
Then the proposition holds since taking the Hilbert series of a cone is a valuation.
\end{proof}

As a corollary, for any linear map $f:K^0_T(G(d,n))\to A$, 
the composition $f\circ y$ is a valuation as well.  
In particular, all of the following are matroid valuations:
the product of $y(M)$ with a fixed class $\eqv{[E]}\in K^0_T(G(d,n))$;
any pushforward of such a product; and the non-equivariant version of any of these
obtained by sending all characters of~$T$ to~1.  
Note that $S_n$ acts trivially on $K^0(G(d,n))$, so $M \mapsto y(M)$ is a matroid invariant, 
and so is $M \mapsto \int y(M) [E]$ for any $E \in K^0(G(d,n))$.
On the other hand, $S_n$ acts nontrivially on $K^0_T(G(d,n))$, so valuations built from 
equivariant $K$-theory need not be matroid invariants.  

As the reader can see from Theorem~\ref{thm:DF}, $\IndMat/S_n$ is free of rank $\binom{n}{d}$. 
The group $K^0(G(d,n))$ is also free of rank $\binom{n}{d}$. 
This gives rise to the hope that every valuative matroid invariant 
might factor through $M\mapsto y(M)$, i.e. that every matroid valuation might come from $K$-theory.  This hope is quite false. 
We give a conceptual explanation for why it is wrong, followed by a counterexample.

The reason this should be expected to be false is that no torus orbit closure can have dimension greater than that of $T$, namely $n-1$.
Therefore, $\int y(M) [E]$ vanishes whenever $E$ is supported in codimension $n$ or greater.
This imposes nontrivial linear constraints on $y(M)$, so the classes $y(M)$ span a proper subspace of $K^0(G(d,n))$.

\begin{example}
We exhibit an explicit non-$K$-theoretic valuative matroid invariant.
Up to isomorphism, there are $7$ matroids of rank $2$ on $[4]$.
Six of them are of the form $\SM(I)$; the last is $M_0:=\{ 13, 23, 14, 24 \}$. 
The unique linear relation in $\IndMat/S_4$ is 
$$[M_0]  = 2 [\SM(13)] - [\SM(12)],$$
corresponding to the unique matroid polytope subdivision of these matroids, 
an octahedron cut into two square pyramids along a square.
However, in $K^0(G(2,4))$, we have the additional relation 
$$y(M_0) = y(\SM(14)) + y(\SM(23)) - y(\SM(24)).$$
The reader can verify this relation by using equivariant localization to express
$y(M_0)$ as a $\KTP$-linear combination of the $y(SM(I))$, and then applying Theorem~\ref{thm:K0TtensorQ}.

Consider the matroid invariant where $z(\SM(12))=z(\SM(13))=z(M_0)=1$ and $z(\SM(I))=0$ for all other $I$ (extended to be $S_4$-invariant in the unique way).%
\footnote{The reader may prefer the following description: $z(M)$ is $1$ if $\Poly(M)$ contains  $(1/2, 1/2, 1/2, 1/2)$ and $0$ otherwise.}
Then $z$ is valuative, but does not extend to a linear function on $K^0(G(d,n))$.
\end{example}

\section{A fundamental computation} \label{sec:pushlemma}

Let $[E]$ be a class in $K^0(G(d,n))$.
Recall from section~\ref{sec:intro}
the maps $\pi_d : \Fl(1,d,n-1;n)  \to G(d,n)$ 
and $\pi_{1(n-1)}: \Fl(1,d,n-1; n) \to \P^{n-1} \times \P^{n-1}$,
and the notations $\alpha$ and $\beta$ for the hyperplane classes in $K^0(\P^{n-1} \times \P^{n-1})$. 

Over $G(d,n)$, we have the tautological exact sequence
\begin{equation}
0 \to S \to \CC^n \to Q \to 0. \label{KeySequence}
\end{equation}
Over each point of $G(d,n)$, the fiber of $S$ is the corresponding $d$-dimensional vector space.

The point of this section is the following computation:
\begin{lemma} \label{pushtoproj}
Given $[E] \in K^0(G(d,n))$, define a formal polynomial in $u$ and $v$ by
$$R(u,v) := \int_{G(d,n)} [E] \sum [\bw^p S] [\bw^q (Q^{\vee})] u^p v^q.$$
Then
$$(\pi_{1 (n-1)})_* \pi_d^* [E] = R(\alpha-1, \beta-1).$$
\end{lemma}

\begin{remark} 
Lemma~\ref{pushtoproj} is an equality in non-equivariant $K$-theory.
In equivariant $K$-theory, we may only speak of the class of a hyperplane if it is a coordinate hyperplane, and then the class depends on which coordinate hyperplane it is. 
We do not have an equivariant generalization of Lemma~\ref{pushtoproj}. 
\end{remark}

For the purposes of this section we will write 
$\kappa=[\O(1,0)]$ and $\lambda=[\O(0,1)]$.
Recall that $\kappa^{-1}=1-\alpha$ and $\lambda^{-1}=1-\beta$,
by exact sequence~\eqref{eq:OOH}. 
For $k$, $\ell \geq 0$, we have 
$(\pi_d)_* (\pi_{1(n-1)})^* (\kappa^k \lambda^{\ell}) = [\Sym^k S^{\vee} \otimes \Sym^{\ell} Q]$.

From the sequence~(\ref{KeySequence}), 
we have $[S] + [Q] = n$. 
Similarly, we have a filtration 
$0\subseteq\bw^k S=F_n\subseteq F_{n-1}\subseteq\cdots\subseteq F_0=\bw^k\CC^n$,
where $F_i$ is spanned by wedges $i$ of whose terms lie in $S$.
Its successive quotients are $\bw^k S$, $\bw^{k-1}S\otimes Q$, \ldots, $\bw^k Q$,
giving the relation $\sum_{i=0}^k [\bw^iS][\bw^{k-i}Q] = \binom nk$.
We can encode all of these relations as a formal power series in $u$, with coefficients in $K^0(G(d,n))$:
 $$ \left( \sum_p [\bw^p(S)] u^p \right) \left( \sum_{\ell} [\bw^{\ell}(Q)] u^{\ell} \right) = (1+ u)^n$$

Also, from the exactness of the Koszul complex~\cite[appendix~A2.6.1]{Eisenbud}, 
$$\left( \sum_k (-1)^k [\Sym^k(Q)] u^k \right) \left( \sum_{\ell} [\bw^{\ell}(Q)] u^{\ell} \right) =1.$$

So
$$\sum [\bw^p(S)] u^p  =   (1+ u)^n \left( \sum (-1)^k [\Sym^k(Q)] u^k \right) .$$

The right hand side is
$$\left((\pi_d)_* \pi_{1(n-1)}^* \sum (-1)^k \kappa^k u^k \right) (1+u)^n= (1+u)^n (\pi_d)_* \pi_{1(n-1)}^* \left( \frac{1}{1+u \kappa} \right). $$

Similarly,
$$\sum [\bw^q(Q)^{\vee}] v^q  = (1+v)^n (\pi_d)_* \pi_{1(n-1)}^* \left( \frac{1}{1+v \lambda} \right). $$

So,
$$R(u,v) = (1+u)^n (1+v)^n \int_{G(d,n)} [E] (\pi_d)_* \pi_{1(n-1)}^* \left( \frac{1}{(1+u \kappa)(1+v \lambda)} \right).$$

By the projection formula (equation~\eqref{ProjFormula}),
$$R(u,v) =  (1+u)^n (1+v)^n  \int_{\PP^{n-1} \times \PP^{n-1}} \left( \vphantom{\frac{1}{1}} (\pi_{1(n-1)})_* \pi_{d}^* [E] \right) \frac{1}{(1+u \kappa)(1+v \lambda)}.$$

Since $\kappa=(1-\alpha)^{-1}$ and $\lambda=(1-\beta)^{-1}$, we get
$$R(u,v) =  \int \left( \vphantom{\frac{1}{1}} \pi_{1(n-1)})_* \pi_{d}^* [E] \right)  \frac{(1+u)^n (1+v)^n}{(1+u (1-\alpha)^{-1}) (1+v (1-\beta)^{-1})}.$$

The quantity multiplying $(\pi_{1(n-1)})_* \pi_{d}^* [E]$ can be expanded as a geometric series
$$\sum (1-\alpha)(1-\beta) \alpha^k \beta^{\ell} (1+u)^{n-1-k}(1+v)^{n-1-\ell}.$$
The sum is finite because $\alpha^n=\beta^n=0$. 

Let $((\pi_{1(n-1)})_* \pi_{d}^* [E]) = \sum T_{ij} \alpha^i \beta^j$. 
Now, $\int_{\PP^{n-1} \times \PP^{n-1}} \alpha^i \beta^j$ is $1$ if $i$ and $j$ are both less than $n$, and zero otherwise.
So 
$$\int \alpha^i \beta^j (1-\alpha)(1-\beta)\alpha^k \beta^{\ell} = \begin{cases} 1 &\mbox{if}\ i=n-1-k\  \mbox{and}\ j=n-\ell-1 \\
0 &\mbox{otherwise} \end{cases}.$$

We deduce that
$$R(u,v) = \sum T_{ij} (1+u)^i (1+v)^j \ \mbox{and} \ R(u-1,v-1)=\sum T_{ij} u^i v^j.$$

Looking at the definition of the $T_{ij}$, we have deduced Lemma~\ref{pushtoproj}.

\section{Flipping Cones}\label{sec:flipping}

Let $f$ be a rational function in $\QQ(z_1, z_2, \ldots, z_n)$. It is possible that many different Laurent power series represent $f$ on different domains of convergence. In this section, we will study this phenomenon.
We recommend~\cite{Ba} as a general introduction to generating functions for lattice points in cones.
The results here can be thought of as generalizations of the relationships between the lattice point enumeration formulas of Brianchon-Gram, Brion and Lawrence-Varchenko.
We recommend \cite{BHS} as an introduction to these formulas.
The reader may also want to consult \cite{Hasse}, which proves some lemmas similar to ours.

Let $\Pow_n$ be the vector space of real-valued functions on $\ZZ^n$ which are linear combinations of the characteristic functions of finitely many lattice polytopes. If $P$ is a pointed polytope, then the sum $\sum_{e \in P} z^e$ converges somewhere, and the value it converges to is a rational function in $\QQ(z_1, \ldots, z_n)$ which we denote $\hilb(P)$.

It is a theorem of Lawrence \cite{Law}, and later Khovanski-Pukhlikov \cite{KP}, that $\One(P) \mapsto \hilb(P)$ extends to a linear map $\hilb: \Pow \to \QQ(z_1, \ldots, z_n)$. 
If $P$ is a polytope with nontrivial lineality space then $\hilb(\One(P))=0$.

\begin{lemma} \label{lem:ConesSpan}
The vector space $\Pow_n$ is spanned by the classes of simplicial cones.
\end{lemma}

\begin{proof}
Let $P$ be any polytope. By the Brianchon-Gram formula~(\cite{Br}, \cite{G}; see also \cite{She} for a modern exposition), $[P]$ is a linear combination of classes of cones. We can triangulate those cones into simplicial cones.
\end{proof}

Let $\zzeta := (\zeta_1, \zeta_2, \ldots, \zeta_n)$ be a basis for $\RR^n$,
which is given the standard inner product.
Define an order $<_{\zzeta}$ on $\QQ^n$ by $x <_{\zzeta} y$ if, for some index $i$, we have $\langle \zeta_1, x \rangle = \langle \zeta_1, y \rangle$,  $\langle \zeta_2, x \rangle = \langle \zeta_2, y \rangle$, \dots,  $\langle \zeta_{i-1}, x \rangle = \langle \zeta_{i-1}, y \rangle$ and  $\langle \zeta_i, x \rangle < \langle \zeta_i, y \rangle$. 

\begin{remark}
Note that, if the components of $\zeta_1$ are linearly independent over $\QQ$, we can disregard the later vectors in $\zzeta$.
For any finite collection of vectors in $\QQ^n$, we can find $\zeta'_1$ of this form so that $<_{\zzeta}$ and $<_{\zeta'_1}$ agree on this collection.
We could use this trick to reduce to the case of a single vector in all of our applications, but the freedom to use vectors with integer entries will be convenient.
\end{remark}

We'll say that a polytope $P$ is \newword{$\zzeta$-pointed} if, for every $a \in \RR^n$, the intersection $P \cap \{ e : e <_{\zzeta} a \}$ is bounded.  
We'll say that an element in $\Pow_n$ is $\zzeta$-pointed if it is supported on a finite union of $\zzeta$-pointed polytopes. 
Let $\Pow_n^{\zzeta}$ be the vector space of $\zzeta$-pointed elements in $\Pow_n$.

\begin{lemma}  \label{PowSerUnique}
The restriction of $\hilb$ to $\Pow_n^{\zzeta}$ is injective.
\end{lemma}

\begin{proof}[Proof of Lemma~\ref{PowSerUnique}]
Suppose, for the sake of contradiction, that $\hilb(b)=0$ for some nonzero $b \in \Pow_n^{\zzeta}$.
Note that $\Pow_n$ is a $\QQ[t_1, \ldots, t_n]$ module with the multiplication $t_i * \One(P) = \One(P+e_i)$.
For any simplicial cone $C$, there is a nonzero polynomial $q(t) \in \QQ[t_1, \ldots, t_n]$ such that $q * \One(C)$ has finite support.  Explicitly, we can take $q(t)=\prod (1-t^e)$ where the product is over the minimal lattice vectors on the rays of $C$ \cite[Theorem 4.6.11]{Stanley}. 
So, by Lemma~\ref{lem:ConesSpan}, we can find a nonzero $q \in \QQ[t_1, \ldots, t_n]$ such that $q * b$ is finitely supported. 

Now, $\hilb$ is clearly $\QQ[t_1, \ldots, t_n]$-linear. 
So $\hilb(q*b) = q \cdot \hilb(b) =0$. 
But $q*b$ is finitely supported, so $q*b=0$.

We now use that $b$ is $\zzeta$-pointed. 
Let $e$ be the $\zzeta$-minimal element of $\ZZ^n$ for which $b(e) \neq 0$.
Also, let $d$ be the $\zzeta$-minimal exponent for which $t^d$ occurs in $q$.
Then the coefficient of $d+e$ in $q*b$ is nonzero, a contradiction.
\end{proof}

We will usually use the above lemma in the following, obviously equivalent, form:

\begin{cor} \label{ExtractCoeff}
Suppose that we have functions $f_1$, $f_2$, \dots, $f_r$, $g_1$, $g_2$, \dots, $g_s$ in $\Pow_n^{\zzeta}$ and scalars $a_1$, \dots, $a_r$, $b_1$, \dots, $b_s$ such that $\sum a_i \hilb(f_i)=\sum b_j \hilb(g_j)$. Let $e$ be any lattice point in $\ZZ^n$.
Then $\sum a_i f_i(e)=\sum b_j g_j(e)$.
\end{cor}

Let $C$ be a simplicial cone with vertex $w$, spanned by rays $v_1$, $v_2$, \dots, $v_r$. 
Reorder the $v_i$ so that $v_i <_{\zzeta} 0$ for $1 \leq i \leq \ell$ and $v_i >_{\zzeta} 0$ for $\ell+1 \leq i \leq r$. 
Define the set $C^{\zzeta}$ to be
$$C^{\zzeta} = \{ w + \sum_{i=1}^r a_i v_i : \ a_i < 0 \ \mbox{for $1 \leq i \leq \ell$ and} \ a_i \geq 0 \ \mbox{for $\ell+1 \leq i \leq n$}\}$$
and define 
$$\One(C)^{\zzeta} = (-1)^{\ell} \One (C^{\zzeta}).$$

Note that $C^{\zzeta}$ is $\zzeta$-pointed.

\begin{lemma} \label{FlipWorks}
 With the above notation,
$$\hilb(\One(C)) = \hilb(\One(C)^{\zzeta}).$$
\end{lemma}

An example of Lemma~\ref{FlipWorks} is that $\sum_{i \geq 0} z^i$ and $- \sum_{i<0} z^i$ both converge to $1/(1-z)$, on different domains.

\begin{proof}
For $I$ a subset of $\{ 1,2, \ldots, \ell \}$, set 
$$C_I :=  \{ w + \sum_{i=1}^r a_i v_i : a_i \geq 0 \ \mbox{for $i \not \in I$},\ a_i \in \RR\ \mbox{for $i \in I$} \}.$$
So $C_{\emptyset} = C$. Then
$$\One(C)^{\zzeta} = \sum_{I \subset [\ell]} (-1)^{|I|} \One(C_I).$$
Applying $\hilb$ to both sides of the equation, all the terms drop out except
$$\hilb(\One(C)^{\zzeta}) = \hilb(\One(C_{\emptyset})) = \hilb(\One(C)).$$
\end{proof}

The following lemma, in the case that $\zeta_1$ has linearly independent components over $\QQ$, is the main result of \cite{Hasse}. 

\begin{lemma} \label{FlipDefined}
Let $\zzeta = (\zeta_1, \ldots, \zeta_n)$ be as above.
For every $f \in \Pow_n$, there is a unique $f^{\zzeta} \in \Pow_n^{\zzeta}$ such that $\hilb(f)=\hilb(f^{\zzeta})$.
The map $f \mapsto f^{\zzeta}$ is linear.
\end{lemma}

By Lemma~\ref{FlipWorks}, this notation $f^{\zzeta}$ is consistent with the earlier notation $\One(C)^{\zzeta}$.

\begin{proof}
We get uniqueness from Lemma~\ref{PowSerUnique}. By Lemma~\ref{lem:ConesSpan}, it is enough to show $\One(D)^{\zzeta}$ exists for $D$ a simplicial cone. 
This is Lemma~\ref{FlipWorks}.

Finally, we must establish linearity.
Let $f$ and $g \in \Pow_n$ and let $a$ and $b$ be scalars. 
Then 
\begin{multline*}
\hilb( (af+bg)^{\zzeta}) = \hilb(af+bg) = a \hilb(f) + b \hilb(g) = \\ a \hilb(f^{\zzeta}) + b \hilb(g^{\zzeta}) = \hilb( a (f^{\zzeta}) + b (g^{\zzeta})).
\end{multline*}
By uniqueness, we must have $(af+bg)^{\zzeta} = a (f^{\zzeta}) + b (g^{\zzeta})$.
\end{proof}

\begin{remark}
We warn the reader that, when $C$ is not simplicial, $\One(C)^{\zzeta}$ need not be of the form $\pm \One(C')$. For example, let $C$ be the span of $(0,0,1)$, $(1,0,1)$, $(0,1,1)$ and $(1,1,1)$. 
Choose $\zeta_1$ to be negative on $(0,0,1)$, $(1,0,1)$ and positive on  $(0,1,1)$ and $(1,1,1)$. Then $\One(C)^{\zzeta} = \One(U)-\One(V)$ where  
$$U = \{ a(1,0,0) + b(0,0,-1) + c(0,-1,-1) : a \geq 0, \ b,c > 0 \}$$ 
and 
$$V=\{ a(1,0,0) + b(1,0,1) + c(1,1,1) : a>0, b,c \geq 0 \}.$$
\end{remark}

\begin{lemma} \label{ZetaVanishing}
Let $C$ be a pointed cone with vertex at $w$. Then $\One(C)^{\zzeta}$ is contained in the half space $\{ x : \langle \zeta_1, x \rangle \geq \langle \zeta_1, w \rangle \}$. Furthermore, if $C$ is not contained in $\{ x : \langle \zeta_1, x \rangle \geq \langle \zeta_1, w \rangle \}$, then $\One(C)^{\zzeta}$ is in the open half space  $\{ x : \langle \zeta_1, x \rangle > \langle \zeta_1, w \rangle \}$.
\end{lemma}

\begin{proof}
For simplicial cones, this follows from the explicit description of $\One(C)^{\zzeta}$ in Lemma~\ref{FlipWorks}. Since any cone can be triangulated, the statement about the closed half space follows immediately from linearity and the simplicial case.

If $C$  is not contained in $\{ x : \langle \zeta_1, x \rangle \geq \langle \zeta_1, w \rangle \}$ then there is some ray of $C$ in direction $v$ with $\langle \zeta_1, v \rangle <0$. Choose a triangulation of $C$ in which every interior face uses the ray $v$.
For example, we can triangulate the faces of $C$ which do not contain $v$, then cone that triangulation from $v$.
(This is called a \newword{pulling triangulation}.)

Letting $\mathcal{F}$ be the set of interior cones of this triangulation, we have $\One(C) = \sum_{F \in \mathcal{F}} (-1)^{\dim C - \dim F} \One(F)$ and 
$\One(C)^{\zzeta} = \sum_{F \in \mathcal{F}} (-1)^{\dim C - \dim F} \One(F)^{\zzeta}$. By the simplicial computation, each summand on the right is supported on the required open half space.
\end{proof}

\begin{cor} \label{ConvexHull}
Let $C_i$ be a finite sequence of pointed cones in $\RR^n$, with the vertex of $C_i$ at $w_i$.
Let $a_i$ be a finite sequence of scalars.
Suppose that we know $\sum a_i \hilb(C_i)$ is a Laurent polynomial.
Then its Newton polytope is contained in the convex hull of the $w_i$.
\end{cor}

\begin{proof}
Let $P$ be the Newton polytope in question and let $\sum_{e \in P} f(e) z^e$ be the polynomial.
Extend $f$ to $\ZZ^n$ by $f(e)=0$ for $e \not \in P$. Since $P$ is a bounded polytope, $f$ is $\zzeta$-pointed for every $\zzeta$ and, thus, $f^{\zzeta} = f$ for every $\zzeta$.

Let $e$ be a lattice point which is not contained in the convex hull of the $w_i$. 
By the Farkas lemma \cite[Proposition 1.10]{Zie}, there is some $\zeta_1$ such that $\langle \zeta_1, e \rangle < \langle \zeta_1, w_i \rangle$ for all $i$.
Complete $\zeta_1$ to a basis $\zzeta$ of $\RR^n$.
For this $\zzeta$, Lemma~\ref{ZetaVanishing} shows that $f^{\zzeta}$ does not contain $e$. 
But, as noted above, $f^{\zzeta} = f$. So $f(e)=0$. We have shown that $f(e)=0$ whenever $e$ is not in the convex hull of the $w_i$, which is the required claim.
\end{proof}

\section{Proof of Theorem~\ref{TutteTheoremProved}} \label{sec:tutte}

Let $M$ be a rank $d$ matroid on the ground set $[n]$, and let $\rho_M$ be the rank function of $M$. The rank generating function of $M$ is 
$$r_M(u,v) := \sum_{S \subset [n]} u^{d - \rho_M(S)} v^{|S| - \rho_M(S)}.$$
The Tutte polynomial is defined by $t_M(z,w) = r_M(z-1,w-1)$. 
See \cite{BO} for background on the Tutte polynomial, including several alternate definitions.

We continue to use the notations $\pi_d$, $\pi_{1(n-1)}$, $\alpha$ and $\beta$ 
from section~\ref{sec:intro}, and the notation $\KTP$ for $K^0(\pt) = \ZZ[t_1^{\pm 1}, \ldots, t_n^{\pm 1}]$.

The aim of this section is to prove:

\begin{theorem} \label{TutteTheoremProved}
We have
$$(\pi_{1(n-1)})_* \pi_{d}^* \left( y(M) \cdot  [\mathcal{O}(1)] \vphantom{\int} \right) = t_M(\alpha, \beta).$$
\end{theorem}
As usual, the sheaf $\OO(1)$ on $G(d,n)$ is the pullback of~$\OO(1)$ on $\PP^N$ via the Pl\"ucker embedding.
We can also describe $\OO(1)$ as  $\bw^d S^{\vee}$.

By Lemma~\ref{pushtoproj}, it is enough to show instead that
$$\int y(M) \cdot [\mathcal{O}(1)]  \cdot \sum_{p=0}^d \sum_{q=0}^{n-d} \ [\bw^p S] \ [\bw^q (Q^{\vee})] u^p v^q = r_M(u,v).$$
In fact, we will show something stronger.  

\begin{theorem}\label{thm:RankCorankGF}
In equivariant $K$-theory, we have
\begin{multline}
\int^T \sum_{p=0}^d \sum_{q=0}^{n-d} y(M) \  \eqv{[\mathcal{O}(1)]}\  \eqv{[\bw^p S]} \ \eqv{[\bw^q (Q^{\vee})]} \ u^p v^q \\
= \sum_{ S \subset [n] } t^{e_S} u^{d-\rho_M(S)} v^{ |S| - \rho_M(S)}. \label{TutteCheckThisEquivariant}  
\end{multline}
\end{theorem}

That is, the integral \eqref{TutteCheckThisEquivariant}
is a generating function in $\KTP[u,v]$ recording the subsets of $[n]$ 
which $r_M(u,v)$ enumerates.

As defined earlier, let $e_S = \sum_{i \in S} e_i$.
We now begin computing the left hand side of~(\ref{TutteCheckThisEquivariant}), using localization. 
Let $I \in \binom{[n]}{d}$ and abbreviate $[n] \setminus I$ by $J$.  Because the localization
of a vector bundle at $x_I$ is the character of its stalk there, we have
\begin{align*}
\left( \vphantom{[E]^{T^2}} \eqv{[\mathcal{O}(1)]} \eqv{[\bw^p S]} \eqv{[\bw^q (Q^{\vee})]} \right)(I) 
&= (t_{i_1} \cdots t_{i_d}) e_p(t_i^{-1})_{i \in I} e_q(t_j)_{j \in J} 
\\&= e_{d-p}(t_i)_{i \in I} e_q(t_j)_{j \in J}
\end{align*}
where $e_k$ is the $k$-th elementary symmetric function.  
Summing over $p$ and $q$ and expanding, we get
$$\sum_{p=0}^d \sum_{q=0}^{n-d} \eqv{[\mathcal{O}(1)]} \eqv{[\bw^p S]} \eqv{[\bw^q (Q^{\vee})]}(I)
= \sum_{P \subseteq I} \sum_{Q \subseteq J} t^{e_P+e_Q} u^{d-|P|} v^{|Q|}.$$
So we want to compute
\begin{equation}
\sum_{I \in M} \hilb(\Cone_{I}(M)) \sum_{P \subseteq I} \sum_{Q \subseteq J} t^{e_P + e_Q} u^{d-|P|} v^{|Q|}. \label{TutteExpanded}
\end{equation}
The reader may want to consult Example~\ref{SquarePyramid2} at this time.

Although by its looks this sum is a rational function in the $t_i$,
it is a class in $\KTP$ and is therefore a Laurent polynomial. 
By Corollary~\ref{ConvexHull},  all the exponents appearing with nonzero coefficient in this polynomial must be in the convex hull of the set of all exponents which can be written as $e_P + e_Q$, for $P$ and $Q$ as above. Since $P$ and $Q$ are disjoint, all of these exponents are in the cube $\{ 0,1 \}^n$, so the polynomial~(\ref{TutteExpanded}) must be supported on monomials of the form $t^{e_S}$. Fix a subset $S$ of $[n]$. Our goal is now to compute the coefficient of $t^{e_S}$ in~(\ref{TutteExpanded}). 

Choose $\zeta_1\in\RR^n$ such that the components of $\zeta_1$ are linearly independent over $\mathbb{Q}$, the component $(\zeta_1)_i$ is negative for $i \in S$ and $(\zeta_1)_i$ is positive for $i \not \in S$.
Clearly, on the cube $\{ 0,1 \}^n$, the minimum value of $\zeta_1$ occurs at $e_S$.
Complete $\zeta_1$ to a basis $\zzeta$ of $\RR^n$.
Note that $\zeta_1$ assumes distinct values on the $2^n$ points of the unit cube.
Then~(\ref{TutteExpanded}) is equal to
\begin{equation}
\sum_{I \in M} \hilb( \One(\Cone_{I}(M))^{\zzeta} ) \sum_{P \subseteq I} \sum_{Q \subseteq J} t^{e_{P \cup Q}} u^{d-|P|} v^{|Q|}. \label{TutteExpandedFlipped}
 \end{equation}
By Corollary~\ref{ExtractCoeff} we can compute the coefficient of $t^{e_S}$ in this polynomial by adding up the coefficients of $t^{e_S}$ in each term.

We therefore consider the coefficient of $t^{e_S}$ in $t^{e_{P \cup Q}} \hilb( \One(\Cone_{I}(M)^{\zzeta}) )$. 
The function $e_{P \cup Q} + \One(\Cone_I(M))^{\zzeta}$ is supported on a cone whose tip is at $e_{P \cup Q}$, and which is contained in the half space $\{ x: \langle \zeta_1, x\rangle \geq \langle \zeta_1, e_{P \cup Q} \rangle \}$.
Since $e_{P \cup Q}$ is in the unit cube $\{ 0,1 \}^n$, our choice of $\zeta_1$ implies that $\langle \zeta_1, e_S \rangle \leq \langle \zeta_1, e_{P \cup Q} \rangle$. 
So $t^{e_{P \cup Q}} \hilb( \One(\Cone_{I}(M))^{\zzeta} )$ contains a $t^{e_S}$ term only if $S = P \cup Q$. 

Even if $S=P \cup Q$, by Lemma~\ref{ZetaVanishing}, the coefficient of $t^{e_S}$ is nonzero only if $\Cone_{I}(M)$ is in the half space where $\zeta_1$ is nonnegative. This occurs if and only if $\zeta_1(e_I) \leq \zeta_1(e_{I'})$ for every $I' \in M$.

In short, the coefficient of $t^{e_S}$ receives nonzero contributions from those triples $(I, P, Q)$ such that
\begin{enumerate}
 \item The function $\zeta_1$, on $\Poly(M)$, is minimized at $e_I$.
\item $P  \subseteq I$ and $Q \subseteq [n] \setminus I$.
\item $S = P \cup Q$.
\end{enumerate}
The contribution from such a triple is $u^{d-|P|} v^{|Q|}$. 

Because $\zeta_1$ takes distinct values on $\{ 0,1 \}^n$, there is only one basis of $M$ at which $\zeta_1$ is minimized. 
Call this basis $I_0$.
Moreover, there is only one way to write $S$ as $P \cup Q$ with $P \subseteq I_0$ and $Q \subseteq [n] \setminus I_0$; we must take $P = S \cap I_0$ and $Q = S \cap ([n] \setminus I_0)$. 
So the coefficient of $t^{e_S}$ is $u^{d - |S \cap I_0|} v^{|S \cap ([n] \setminus I_0)|}$.

From the way we chose $\zeta_1$, we see that $I_0$ is an element of $M$ with maximal intersection with $S$.
In other words, $|S \cap I_0| = \rho_M(S)$. From the description in the previous paragraph, the coefficient of $t^{e_S}$ is $u^{d - \rho_M(S)} v^{|S| - \rho_M(S)}$.
Summing over $S$, we have equation~(\ref{TutteCheckThisEquivariant}), 
and Theorems \ref{TutteTheoremProved} and~\ref{thm:RankCorankGF} are proved.

\begin{question}
Is there an equivariant version of Lemma~\ref{pushtoproj} which provides 
a generating function in $\KTP[u,v]$ for the bases of given activity,
parallel to Theorem~\ref{thm:RankCorankGF} for the rank generating function?
\end{question}

\begin{example} \label{SquarePyramid2}
We compute the sum in~(\ref{TutteExpanded}) for the matroid from Example~\ref{SquarePyramid1}. 
We can shorten our expressions slightly by defining 
\begin{eqnarray*}
s_{I} &:=& \sum_{P \subseteq I} \sum_{Q \subseteq J} t^{e_P + e_Q} u^{d-|P|} v^{|Q|} = \prod_{i \in I} (u+t_i) \prod_{j \in J} (1+v t_j) \\
h_{I} &:=& \prod_{i \in I} \prod_{j \in J} (1-t_i^{-1} t_j)^{-1}. 
\end{eqnarray*}

Then (\ref{TutteExpanded}) is 
\begin{multline*}
s_{13} h_{13} (1-t_2 t_3^{-1}) + s_{14} h_{14} (1-t_2 t_4^{-1}) + s_{23} h_{23} (1-t_1 t_3^{-1})  \\ + s_{24} h_{24} (1-t_2 t_4^{-1}) + s_{34} h_{34} (1-t_1 t_2 t_3^{-1} t_4^{-1}), 
\end{multline*}
which is
\begin{multline*}
\left( t_1 t_3 + t_2 t_3 + t_1 t_4 + t_2 t_4 + t_3 t_4 \right) 
+\left(t_1 + t_2 + t_3 + t_4 \right) \cdot u \\
+\left( t_1 t_2 t_3  + t_1 t_2 t_4 + t_1 t_3 t_4 + t_2 t_3 t_4 \right) \cdot v
+ u^2 + \left(t_1 t_2 \right) \cdot uv  + \left( t_1 t_2 t_3 t_4 \right) \cdot v^2
\end{multline*}

Specializing the $t_i$ to zero gives the rank-generating function
$$5+4u+4v+u^2+uv+v^2.$$
Setting $u=z-1$ and $v=w-1$ gives the Tutte polynomial
$$w + z+ w^2 + w z + z^2.$$


\end{example}

\section{Proof of Theorems \ref{DiagonalPolynomial}, \ref{OldNew}} \label{sec:h}

In this section, we discuss the relation between localization methods and the matroid invariant $h_M$ discovered by the second author. Our first aim is to prove Theorem~\ref{DiagonalPolynomial} below,
defining a polynomial $H_M$.  We will then discuss the relation of $H_M$ to $h_M$.

\begin{theorem} \label{DiagonalPolynomial}
Let $M$ be a rank $d$ matroid on $[n]$ without loops or coloops. Let the maps $\pi_d$ and $\pi_{1(n-1)}$ 
and the classes $\alpha$ and $\beta$ be as in Section~\ref{sec:intro}.  
Then there exists a polynomial $H_M\in\ZZ[s]$ such that
$$(\pi_{1(n-1)})_* \pi_{d}^*\, y(M) = H_M(\alpha+\beta-\alpha \beta).$$
\end{theorem}

Because $(\alpha+\beta - \alpha \beta)^n=0$, there is more than one polynomial which obeys this condition. 
We make $H_M$ unique by defining it to have degree $<n$.

The heart of our proof is the following lemma:

\begin{lemma}\label{lem:Zcoeffs}
In the setup of Theorem~\ref{DiagonalPolynomial},
$\int^T y(M) \eqv{[\bw^p S]} \eqv{[\bw^q (Q^{\vee})]}\in\ZZ$ for any $p$ and~$q$, and equals $0$ when $p\neq q$.  
\end{lemma}

\begin{proof}[Proof of Theorem~\ref{DiagonalPolynomial} from Lemma~\ref{lem:Zcoeffs}]

Suppose that $(\pi_{1(n-1)})_* \pi_{d}^* \left( y(M) \right)=F(\alpha, \beta)$. 
To say that $F$ is a polynomial in $\alpha+\beta-\alpha \beta$ is the same as to say that it is a polynomial in
$1-\alpha+\beta-\alpha \beta=(\alpha-1)(\beta-1)$. So, by Lemma~\ref{pushtoproj}, it is equivalent to show that $\int y(M) \sum [\bw^p S] [\bw^q (Q^{\vee})] u^p v^q$ is a polynomial in $uv$. 
By Lemma~\ref{lem:Zcoeffs}, the coefficient of $u^p v^q$ is zero whenever $p \neq q$, so this sum is a polynomial in $uv$.
\end{proof}

As in the proof of Theorem~\ref{TutteTheoremProved}, the proof of Lemma~\ref{lem:Zcoeffs} will be by equivariant localization.

\begin{proof}
Fix $p$ and~$q$. 
For any $I \in \binom{[n]}{d}$, we have 
$$\eqv{[\bw^p S]} \eqv{[\bw^q (Q^{\vee})]}(I) = e_p(t_i^{-1})_{i \in I} e_q(t_j)_{j \in [n] \setminus I}$$
where $e_k$ is the $k$\/th elementary symmetric function.  
So
\begin{equation}
\int y(M) \eqv{[\bw^p S]} \eqv{[\bw^q (Q^{\vee})]} =
\sum_{I \in M} \hilb(\Cone_I(M)) \sum_{P \in \binom{I}{p}} t^{-e_P} \sum_{Q \in \binom{[n] \setminus I}{q}} t^{e_Q}. \label{VanishingExpanded}
\end{equation}
The reader may wish to consult example~\ref{SquarePyramid3} at this time.

By Corollary~\ref{ConvexHull}, $t^a$ may only appear with nonzero coefficient if $a$ is in the convex hull of $\{ e_P - e_Q \}$ where $P$ and $Q$ are as above. In particular, every coordinate of $a$ must be $-1$, $0$ or $1$. We will now establish that, in fact, every coordinate must be zero.

Consider any index $i$ in $[n]$.
Let $\zeta_1 = e_i$ and complete $\zeta_1$ to a basis $\zzeta$ of~$\RR^n$.
We abbreviate the half space $\{ x : x_i \geq 0 \}$ by $H$, and $\{ x : x_i > 0 \}$ by $H_{+}$.

The sum in~(\ref{VanishingExpanded}) is equal to
\begin{equation}
\sum_{I \in M} \hilb( \One(\Cone_{I}(M))^{\zzeta} ) \sum_{P \in \binom{I}{p}} \sum_{Q \in \binom{[n] \setminus I}{q}} t^{e_Q-e_P}. \label{VanishingExpandedFlipped}
\end{equation}
By Corollary~\ref{ExtractCoeff}, it is legitimate to extract the coefficient of a particular term. 

Let $I\in M$, and suppose that $i \not \in I$. Then $i$ cannot be in $P$, so the $i$-th coordinate in $e_Q-e_P$ is nonnegative. Also, by Lemma~\ref{ZetaVanishing}, $\One(\Cone_I(M))^{\zzeta}$ is supported in $H$. So such $t^{e_Q - e_P} \hilb( \One(\Cone_I(M))^{\zzeta} )$ cannot contribute any monomial of the form $t^a$ with $a_i<0$.

Now, suppose that $i \in I$. 
Since $i$ is not a coloop of $M$, the cone $\Cone_I(M)$ has a ray with negative $i$-th coordinate.
So, by Lemma~\ref{ZetaVanishing}, $\One(\Cone_I(M))^{\zzeta}$ lies in the open halfplane $H_{+}$.
In particular, if $\One(\Cone_I(M))^{\zzeta}(a)$ is nonzero for some lattice point $a$ then $a_i \geq 1$. 
So, again, $t^{e_Q - e_P} \hilb( \One(\Cone_I(M))^{\zzeta} )$ cannot contribute any monomial of the form $t^a$ with $a_i<0$.

A very similar argument shows that no monomial with any positive exponent can occur in~(\ref{VanishingExpandedFlipped}).
So the only monomial in~(\ref{VanishingExpandedFlipped}) is $t^0$, i.e.\
(\ref{VanishingExpandedFlipped}) is in~$\ZZ$.
Additionally,~(\ref{VanishingExpandedFlipped}) is homogeneous of degree $q-p$, which is nonzero
if $p\neq q$.
So we deduce that in that case~(\ref{VanishingExpandedFlipped}) is zero, as desired.
\end{proof}

\begin{example} \label{SquarePyramid3}
We compute $H_M$ for the matroid $M$ from example~\ref{SquarePyramid1}.
For brevity, we write
\begin{eqnarray*}
s'_{I} &:=& \sum_{P \subseteq I} \sum_{Q \subseteq J} t^{- e_P + e_Q} u^{|P|} v^{|Q|} = \prod_{i \in I} (1+u t_i^{-1}) \prod_{j \in J} (1+v t_j) \\
h_{I} &:=& \prod_{i \in I} \prod_{j \in J} (1-t_i^{-1} t_j)^{-1}. 
\end{eqnarray*}
We must compute
\begin{multline} \label{BigMess}
s'_{13} h_{13} (1-t_2 t_3^{-1}) + s'_{14} h_{14} (1-t_2 t_4^{-1}) + s'_{23} h_{23} (1-t_1 t_3^{-1})  \\ + s'_{24} h_{24} (1-t_2 t_4^{-1}) + s'_{34} h_{34} (1-t_1 t_2 t_3^{-1} t_4^{-1}). 
\end{multline}
The reader may enjoy typing~(\ref{BigMess}) into a computer algebra system and watching it simplify to $1-uv$.
So $H_M = 1-(1-\alpha)(1-\beta) = \alpha+\beta-\alpha \beta$ and $h_M(s)=s$.
\end{example}

We now show that the polynomial $H_M$ is equal to the polynomial $h_M$ from the 
second author's earlier work~\cite{KT1}.



\begin{remark}
In~\cite{KT1}, two closely related polynomials are introduced, $h_M(s)$ and $g_M(s)$. These obey
$g_M(s) = (-1)^c h_M(-s)$, where $c$ is the number of connected components of $M$.
As discussed in \cite[Section 3]{KT1}, $g_M$ behaves more nicely in combinatorial formulas; 
its coefficients are positive and formulas involving $g_M$ have fewer signs.
However, $h_M$ is more directly related to algebraic geometry.
The fact that $h_M$ arises more directly in the present paper is another indication of this.
\end{remark}

We review some the definition of $h_M$.
Let $i$ be an index between $1$ and $d$. 
Choose a line $\ell$ in $n$-space and an $n-i$ plane $M$ containing $\ell$. 
Let $\Omega_i \subset G(d,n)$ be the Schubert cell of those $d$-planes $L$ such that $\ell \subset L$ and $L+M$ is contained in a hyperplane.
If $i > d$, we define $\Omega_i$ to be $\Omega_d$.
Then $h_M(s)$ was defined by
$$\frac{h_M(s)}{1-s} = \sum_{i=1}^{\infty} \int_{G(d,n)} y(M) [\mathcal{O}_{\Omega_i}] s^i.$$ 
In other words, the coefficient of $s^i$ in $h_M(s)$ is 
$$\int_{G(d,n)} y(M) \left( [\mathcal{O}_{\Omega_i}] - [\mathcal{O}_{\Omega_{i-1}}] \right).$$

\begin{theorem} \label{OldNew}
With the above definitions, we have $H_M(s) = h_M(s)$.
\end{theorem}

\begin{proof}
We will show that the coefficient of $s^i$ in both cases is the same.
Notice that the coefficient of $s^i$ in $H_M(s)$ will also be the coefficient of $\beta^i$ in $H_M(\alpha+ \beta - \alpha \beta)$. 
As we computed in the proof of Lemma~\ref{pushtoproj}, 
the dual basis to $\alpha^i \beta^j$ is $\alpha^{d-1-i} \beta^{n-d-1-j} (1-\alpha)(1-\beta)$.
In particular, the coefficient of $\beta^i$ in $(\pi_{1(n-1)})_* \pi_{d}^* y(M)$ is $\int \left( (\pi_{1(n-1)})_* \pi_{d}^* y(M) \right) \alpha^{n-1} \beta^{n-1-i} (1-\beta)$.

Now, $\alpha^{n-1}$ intersects the hypersurface $\Fl(1,n-1;n)$ in the set of all pairs $(\mathtt{line}, \mathtt{hyperplane})$ where $\mathtt{line}$ has a given value $\ell$. 
Intersecting further with $\beta^{n-i-1}$ imposes in addition that $\mathtt{hyperplane}$ contain a certain generic $n-i-1$ plane $N$.
But, since the hyperplane is already forced to contain $\ell$, it is equivalent to say that the hyperplane contains the $n-i$ plane $N + \ell$. 
In short, $\alpha^{n-1} \beta^{n-i-1} \cap \Fl(1, n-1; n)$ is represented by the Schubert variety of pairs $(l, H)$ where $l$ is a given line $\ell$ and $H$ contains a given $n-i$ plane $M$ containing $\ell$.

Now, the pushforward of the structure sheaf of a Schubert variety is always the structure sheaf of a Schubert variety.
In the case at hand,
$(\pi_{1(n-1)})_* \pi_{d}^* \alpha^{n-1} \beta^{n-i-1}$ is the class of the Schubert variety of $d$-planes $L$ such that $\ell \subset L$ and $L+M$ is contained in a hyperplane. This is to say, 
$(\pi_{d})_* \pi_{1(n-1)}^* \alpha^{n-1} \beta^{n-i-1} = [ \mathcal{O}_{\Omega_{i}}]$.
Using~(\ref{ProjFormula}), we see that the coefficient of $s^i$ in $H_M(s)$ is 
\begin{multline*}
\int_{\PP^{n-1} \times \PP^{n-1}} \left( (\pi_{1(n-1)})_* \pi_{d}^* y(M) \right) \alpha^{n-1} \beta^{n-1-i} (1-\beta) = \\ \int_{G(d,n)} y(M) \left( [\mathcal{O}_{\Omega_i}] - [\mathcal{O}_{\Omega_{i-1}}] \right), \end{multline*}
as desired.
\end{proof}

\section{Geometric interpretations of matroid operations}\label{sec:matroidops}
In~\cite{KT1}, a number of facts about the behavior of~$h_M$ under standard matroid
operations were proved geometrically.  In this section we re-establish
these using our algebraic tools of localization and Lemma~\ref{pushtoproj}.
Following the established pattern, our proofs will be equivariant.
We first introduce slightly more general polynomials for which our results hold

Define $F^{m}_M(u,v)$ to be the unique polynomial, of degree $\leq n$ in $u$ and $v$, such that
\begin{equation}\label{eq:arbitrary twist}
F^{m}_M(\mathcal{O}(1,0), \mathcal{O}(0,1)) = (\pi_{1(n-1)})_* \pi_{d}^* \left( [\mathcal{O}(m)] y(M) \right).
\end{equation}
We also define an equivariant generalization of this by
$$F^{m,T}_M(u,v) := \int y(M)\eqv{[\mathcal O(m)]}\sum_{p,q}\eqv{[\bw^pS]}\eqv{[\bw^q(Q^\vee)]}u^pv^q$$
In the previous sections, we checked that $F^{0,T}_M(u,v) = h_M(1-uv)$, that $F^{1,T}_M(u,v)$ and $F^{1}_M(u,v)$ are the weighted and unweighted rank generating functions, and that
$F^{1}_M(u-1, v-1)$ is the Tutte polynomial.
The entire collection of $F^{m,T}_M$ can be seen as a generalization of the
Ehrhart polynomial of $\Poly(M)$. 
Specifically, $F^m_M(0,0) = \# (m \cdot \Poly(M) \cap \ZZ^n)$ for $m\geq0$.

Write $M^\ast$ for the matroid dual to~$M$.

\begin{prop}\label{prop:dual}
We have $F^{m}_M(u,v) = F^{m}_{M^\ast}(v,u)\in\ZZ[u,v]$.
\end{prop}

\begin{proof}
Equivariantly, we will show that
$F^{m,T}_M(t) (u,v) = t^{me_{[n]}} F^{m,T}_{M^\ast}(t^{-1}) (v,u)$.
(The symbol $F^{m,T}(t^{-1})$ means that we are to take the coefficients of $F^{m,T}$, which are in $\ZZ[\Char(T)]$, and apply the linear map which inverts each character of $T$.)

We must show that for any $p$ and $q$,
\begin{multline}\label{eq:DualMatroid}
\left(\int^T y(M)\eqv{[\mathcal O(m)]}\eqv{[\bw^pS]}\eqv{[\bw^q(Q^\vee)]}\right)\!(t) 
\\= t^{me_{[n]}}\left(\int^T y(M^\ast)\eqv{[\mathcal O(m)]}\eqv{[\bw^qS]}\eqv{[\bw^p(Q^\vee)]}\right)\!(t^{-1}).
\end{multline}
By localization, the left side is
$$\sum_{I\in M} \hilb(\Cone_I(M))(t)\ t^{me_I} 
\sum_{P\in\binom Ip}\sum_{Q\in\binom{[n]\setminus I}q} t^{e_Q-e_P}.$$
The polytope $\Poly(M^\ast)$ is the image of $\Poly(M)$ under the reflection
$x\mapsto e_{[n]}-x$.
So $\hilb(\Cone_I(M))(t) = \hilb(\Cone_{[n]\setminus I}(M^\ast))(t^{-1})$.  
Therefore the left side of~\eqref{eq:DualMatroid},
reindexing the sum by $J=[n]\setminus I$, is
\begin{align*}
&\mathrel{\phantom{=}}
\sum_{J\in M^\ast} \hilb(\Cone_J(M^\ast))(t^{-1})\ t^{me_{[n]\setminus J}}
\sum_{P\in\binom {[n]\setminus J}p}\sum_{Q\in\binom Jq} t^{e_Q-e_P}
\\&=t^{me_{[n]}}\sum_{J\in M^\ast} \hilb(\Cone_J(M^\ast))(t^{-1})\ t^{-e_J}
\sum_{Q\in\binom Jq}\sum_{P\in\binom {[n]\setminus J}p} t^{-e_P+e_Q}
\end{align*}
which is the right side of~\eqref{eq:DualMatroid}.
\end{proof}

Given matroids $M$ and $M'$, we denote their direct sum by~$M\oplus M'$.

\begin{prop}\label{prop:DirectSum}
We have $F^{m}_{M}F^{m}_{M'} = F^{m}_{M\oplus M'}.$
\end{prop}

\begin{proof}
Localization gives
\begin{equation}\label{eq:direct sum loc}
F^{m}_{M} = \sum_{I\in M}\hilb(\Cone_I(M))\ t^{me_I} 
\sum_{P\subseteq I}\sum_{Q\subseteq\Elts\setminus I} t^{e_Q-e_P}u^{|P|}v^{|Q|}
\end{equation}
and analogous expansions for $M'$ and $M\oplus M'$.
Since $\Poly(M\oplus M') = \Poly(M)\times\Poly(M')$, we have
$$\hilb(\Cone_I(M))\hilb(\Cone_{I'}(M')) = \hilb(\Cone_{I\cup I'}(M\oplus M')).$$
The proposition follows immediately by multiplying out expansions like~\eqref{eq:direct sum loc}.
\end{proof}

For $k=1,2$, let $M_k$ be a matroid on the ground set $\Elts_k$
and let $i_k\in\Elts_k$.  
Consider the larger ground set $\Elts=\Elts_1\sqcup\Elts_2\setminus\{i_1,i_2\}\cup\{i\}$, 
where $i$ should be regarded as the identification of $i_1$ and~$i_2$.
There are three standard matroid operations one can perform in this setting.
In the next definitions, $I_1$ and~$I_2$ range over elements of~$M_1$ and~$M_2$ respectively.
The series connection $M_{\rm ser}$ of $M_1$ and~$M_2$ is the matroid
\begin{align*}
&\{I_1\sqcup I_2: |(I_1\sqcup I_2)\cap\{i_1,i_2\}|=0\}\\
\mbox{}\cup\mbox{}&\{( I_1\sqcup I_2) \setminus\{i_1,i_2\}\cup\{i\}: |(I_1\sqcup I_2)\cap\{i_1,i_2\}|=1\}
\end{align*}
on~$\Elts$; their parallel connection $M_{\rm par}$ is the matroid
\begin{align*}
&\{(I_1\sqcup I_2) \setminus\{i_1,i_2\}: |(I_1\sqcup I_2)\cap\{i_1,i_2\}|=1\}\\
\mbox{}\cup\mbox{}&\{(I_1\sqcup I_2) \setminus\{i_1,i_2\}\cup\{i\}: |(I_1\sqcup I_2)\cap\{i_1,i_2\}|=2\}
\end{align*}
on~$\Elts$; and their two-sum $M_{\rm 2sum}$ is the matroid 
$$\{(I_1\sqcup I_2) \setminus\{i_1,i_2\}: |(I_1\sqcup I_2)\cap\{i_1,i_2\}|=1\}$$
on~$\Elts\setminus\{i\}$. 

The next property has the nicest form for the particular case of~$F^0_M$, 
on account of Lemma~\ref{lem:Zcoeffs}.

\begin{theorem}\label{thm:two-sum}
We have 
$$F^{m}_{M_1\oplus M_2} = (1+v)F^{m}_{M_{\rm ser}}
+ (1+u)F^{m}_{M_{\rm par}}
- (1+v)(1+u)F^{m}_{M_{\rm 2sum}}.$$
In particular, $F^0_{M_{\rm 2sum}} = F^0_{M_{\rm ser}} = F^0_{M_{\rm par}} = F^0_{M_1\oplus M_2}/(1-uv)$.
\end{theorem}

The series, respectively parallel, extension of a matroid $M_1$ along $i_1$ 
is its series, respectively parallel, connection to the uniform matroid $U_{1,2}$.
Two-sum with $U_{1,2}$ leaves $M_1$ unchanged.  Since $H_{U_{1,2}} = 1-uv$,
Proposition~\ref{prop:DirectSum} implies
one of the most characteristic combinatorial properties of~$h$ from~\cite{KT1}.

\begin{cor}
The values of $h_M$, $H_M$ and $F^0_M$ are unchanged by series and parallel extensions.
\end{cor}

\begin{proof}[Proof of Theorem~\ref{thm:two-sum}]
Let $M_k$ have rank~$d_k$, $d=d_1+d_2$, and $n=|\Elts_1|+|\Elts_2|$.
Let $T=(\CC^\ast)^n$ be the torus acting on~$G(d,n)$.
Our aim is to relate $y(M_1\oplus M_2)\in K^0_T(G(d,n))$ to
$y(M_{\rm ser})$, $y(M_{\rm par})$, and~$y(M_{\rm 2sum})$.
Localization renders the problem one of
relating cones at vertices of certain polytopes.  Define 
\begin{align*}
P_{\rm ser} &= \Poly(M_1\oplus M_2)\cap\{x_{i_1}+x_{i_2}\leq 1\} \\
P_{\rm par} &= \Poly(M_1\oplus M_2)\cap\{x_{i_1}+x_{i_2}\geq 1\} \\
P_{\rm 2sum} &= \Poly(M_1\oplus M_2)\cap\{x_{i_1}+x_{i_2} = 1\}
\end{align*}
Then 
$$\One(\Poly(M_1\oplus M_2))=\One(P_{\rm ser})+\One(P_{\rm par})-\One(P_{\rm 2sum}).$$
(If $P_{\rm ser}$ and~$P_{\rm par}$ have the same dimension as~$\Poly(M_1\oplus M_2)$
they will be the facets of a subdivision, with $P_{\rm 2sum}$ the unique other interior face.)
This implies that, for $I\in\binom nd$,
\begin{multline}\label{eq:2sum1}
\hilb(\Cone_I(M_1\oplus M_2)) \\=
  \hilb(\Cone_{e_I}(P_{\rm ser}))
+ \hilb(\Cone_{e_I}(P_{\rm par}))
- \hilb(\Cone_{e_I}(P_{\rm 2-sum})).
\end{multline}

We'll use $L$ to denote one of the symbols ${\rm ser}$, ${\rm par}$, ${\rm 2sum}$.

Let $p:\RR^{\Elts_1\sqcup\Elts_2}\to\RR^\Elts$ be the linear projection
with $p(e_{i_1})=p(e_{i_2})=e_i$ and $p(e_j)=e_j$ for $j\neq i_1,i_2$, and
let $\iota: \RR^{\Elts \setminus \{ i \}} \to \RR^{\Elts}$ be the inclusion into the $i$-th coordinate hyperplane.
Then
\begin{align*}
p(P_{\rm ser}) &= \Poly(M_{\rm ser})\\
p(P_{\rm par}) &= \Poly(M_{\rm par})+e_i\\ 
p(P_{\rm 2sum}) &= \iota(\Poly(M_{\rm 2sum}))+e_i
\end{align*}
where $+e_i$ denotes a translation.

The polytope $\Poly(M_1\oplus M_2)$ lies in the hyperplane $\{\sum_{j\in \Elts_1} x_j=d_1\}$,
which intersects $\ker p$ transversely, so $p$ is an isomorphism on the polytopes $P_L$.
In particular for any $I\in M_1\oplus M_2$ we have
$\Cone_{p(e_I)}(p(P_L))=p(\Cone_{e_I}(P_L))$.  
Also, if $u$ is a lattice point then $p(u)$ is.
Define $r:K^0_T(\pt)\to K^0_{T'}(\pt)$ to be the restriction from characters of~$T$ to characters
of its codimension~1 subtorus 
$$T'=\{(t_j)_{j\in\Elts_1\sqcup\Elts_2}\in T : t_{i_1} = t_{i_2}\},$$
so that $t^{p(e_I)} = r(t^{e_I})$.  We write $t_i$ for the common restriction
of $t_{i_1}$ and $t_{i_2}$ to~$T'$.
We will also occassionally need a notation for the torus $T''$ which is the projection of $T'$ under forgetting the $i$-th coordinate.

Let $A$ be the subring of $\Frac K^0_T(\pt)$ consisting of rational functions whose denominator is not divisible by $t_{i_1} - t_{i_2}$.
The map $r$ extends to a map $r:A\to\Frac K^0_{T'}(\pt)$.
Because $P_L$ is in the hyperplane $\sum_{j \in \Elts_1} x_j = d_1$, the edges of $P_L$ do not point in direction $e_{i_1} - e_{i_2}$, so $\hilb(\Cone_{e_I}(P_L))$ is in $A$ and we have
\begin{align}\label{eq:2sum2}
     \hilb(\Cone_{p(e_I)}(M_L))
  &= \hilb(\Cone_{p(e_I)}(p(P_L))) \notag
\\&= \hilb(p(\Cone_{e_I}(P_L))) \notag 
\\&= r(\hilb(\Cone_{e_I}(P_L)))\,.
\end{align}

We now embark on the computation of $F^{m}_{M_1\oplus M_2}$ by equivariant localization.  We have
\begin{multline*}
F^{m,T}_{M_1\oplus M_2}(u,v) = 
\\ \sum_{I\in M_1\oplus M_2} \hilb(\Cone_I(M_1\oplus M_2))\; t^{me_I}
\sum_{P\subseteq I}\sum_{Q\subseteq\Elts_1\sqcup\Elts_2\setminus I} t^{e_Q-e_P} u^{|P|}v^{|Q|}.
\end{multline*}
Expanding as dictated by~\eqref{eq:2sum1}, this is
\begin{multline}\label{eq:2sumMain}
F^{m,T}_{M_1\oplus M_2}(u,v) = \sum_{I\in M_1\oplus M_2} 
\Big(\hilb(\Cone_{e_I}(P_{\rm ser}))
    + \hilb(\Cone_{e_I}(P_{\rm par})) \\
    - \hilb(\Cone_{e_I}(P_{\rm 2sum})) \Big) \cdot t^{me_I} 
\sum_{P\subseteq I}\sum_{Q\subseteq\Elts_1\sqcup\Elts_2\setminus I} t^{e_Q-e_P} u^{|P|}v^{|Q|} 
\end{multline}

We will eventually be applying the map $\KTP \to K^0(\pt) = \ZZ$
replacing all characters by~1 to get a nonequivariant result.
This map factors through $r$.  
As explained above, all of the terms in equation~\eqref{eq:2sumMain} lie in the ring $A$, so we may apply $r$ to both sides.

We take the three terms inside the large parentheses in~\eqref{eq:2sumMain} individually.  
The three are similar, and we will only work through the first, involving $P_{\rm ser}$,
in detail.  Temporarily denote this subsum $\Sigma_{\rm ser}$, i.e.\
$$\Sigma_{\rm ser} = \sum_{I\in M_1\oplus M_2} \hilb(\Cone_{e_I}(P_{\rm ser}))\; t^{me_I}
\sum_{P\subseteq I}\sum_{Q\subseteq\Elts_1\sqcup\Elts_2\setminus I} t^{e_Q-e_P} u^{|P|}v^{|Q|}.$$
By~\eqref{eq:2sum2} and the definition of~$r$ we have
$$r(\Sigma_{\rm ser}) = \sum_{I\in M_1\oplus M_2} \hilb(\Cone_{p(e_I)}(M_{\rm ser}))\; t^{p(me_I)}
\sum_{P\subseteq I}\sum_{Q\subseteq\Elts_1\sqcup\Elts_2\setminus I} t^{p(e_Q-e_P)} u^{|P|}v^{|Q|}.$$
For any $I\in M_1\oplus M_2$ such that $p(e_I)\in\Poly(M_{\rm ser})$, 
not both $i_1$ and~$i_2$ are in~$I$, so $p(e_I)=e_J$
for some $I'\subseteq\Elts$, and we have
$$\sum_{P\subseteq I}\sum_{Q\subseteq\Elts_1\sqcup\Elts_2\setminus I} t^{p(e_Q-e_P)} u^{|P|}v^{|Q|}
= (1+vt_i)\sum_{P\subseteq I'}\sum_{Q\subseteq\Elts\setminus I'} t^{e_Q-e_P}u^{|P|}v^{|Q|}$$
where the factor $(1+vt_i)$ comes from dropping one of $i_1$ and~$i_2$ not contained
in~$I$ from the sum over~$Q$.  Therefore
\begin{align*}
r(\Sigma_{\rm ser}) &= (1+vt_i)\sum_{I'\in M_{\rm ser}} \hilb(\Cone_{I'}(M_{\rm ser}))\;t^{me_{I'}}
\sum_{P\subseteq I'}\sum_{Q\subseteq\Elts\setminus I'} t^{e_Q-e_P}u^{|P|}v^{|Q|}
\\&= (1+vt_i) F^{m}_{M_{\rm ser}}(u,v).
\end{align*}

A similar argument for each of the other two summands in~\eqref{eq:2sumMain} yields
\begin{multline}\label{eq:2sumLast}
r \left( F^{m,T}_{M_1\oplus M_2}(u,v) \right) = \\(1+vt_i)F^{m,T'}_{M_{\rm ser}}(u,v)
+ (1+ut_i^{-1})F^{m,T'}_{M_{\rm par}}(u,v)
- (1+vt_i)(1+ut_i^{-1})F^{m,T''}_{M_{\rm 2sum}}(u,v).
\end{multline}
In the last term, we are implictly using the injection $K^{T''}_0(\pt) \into K^{T'}_0(\pt)$ coming from the projection $T \to T''$. 

On passing to non-equivariant $K$-theory, this becomes the first assertion of the theorem.
For the second, Lemma~\ref{lem:Zcoeffs} says that $H_M$ is a polynomial in $uv$
for any matroid $M$.  Thus, putting $m=0$ in~\eqref{eq:2sumLast}, the terms on the
right containing an unmatched $v$ must cancel, implying $F^0_{M_{\rm ser}} = F^0_{M_{\rm 2sum}}$.
The same goes for the terms containing an unmatched $u$,
implying $F^0_{M_{\rm par}} = F^0_{M_{\rm 2sum}}$.  Making these substitutions
and simplifying, \eqref{eq:2sumLast} becomes the second assertion of the theorem.
\end{proof}

\end{document}